\documentclass[twocolumn,aps,prx,superscriptaddress,10pt]{revtex4-2}

\usepackage[T1]{fontenc}
\usepackage{amsmath, bm}
\usepackage{graphicx}
\usepackage{amssymb,amsthm}
\usepackage{hyperref}
\usepackage[svgnames]{xcolor}
\usepackage{sidecap}
\usepackage{mathrsfs} 

\newcommand{\by}{\bm{y}}
\newcommand{\bt}{\bm{t}}
\newcommand{\bv}{\bm{v}}
\newcommand{\y}{\by}  
\newcommand{\0}{\bm{0}} 
\newcommand{\f}{\bm{f}}
\newcommand{\w}{\bm{w}}
\newcommand{\bphi}{\boldsymbol{\varphi}}
\newcommand{\btheta}{\boldsymbol{\theta}}
\newcommand{\bxi}{\boldsymbol{\xi}}
\newcommand{\bnu}{\boldsymbol{\nu}}

\newcommand{\D}{\mathcal{D}} 
\newcommand{\M}{\mathcal{M}} 
\newcommand{\K}{\mathcal{N}} 
\newcommand{\V}{\mathcal{P}} 
\newcommand{\T}{\mathcal{T}} 
\newcommand{\U}{\mathcal{U}} 

\newcommand{\bbA}{{\color{blue}{\mathbf{A}}}}  
\newcommand{\bbAt}{{\color{blue}{\widetilde{\mathbf{A}}}}}  
\newcommand{\bbIU}{{\color{red}{\bbI_{\U}}}}  
\newcommand{\bbPN}{{\color{Green}{\mathbf{P}_{\K}}}} 
\newcommand{\bbPNt}{{\color{Green}{\widetilde{\mathbf{P}}_{\K}}}} 
\newcommand{\bbB}{\mathbf{B}}
\newcommand{\bbE}{\mathbf{E}}
\newcommand{\bbI}{\mathbf{I}}
\newcommand{\bbM}{\mathbf{M}}
\newcommand{\bbP}{\mathbf{P}}

\newcommand{\bbX}{\mathbf{X}}

\newcommand{\MTM}{\bbM_{\T\M}}
\newcommand{\tMTM}{\widetilde{\bbM}_{\T\M}}
\newcommand{\MTU}{\bbM_{\T\U}}

\newcommand{\MVM}{\bbM_{\V\M}}
\newcommand{\MVU}{\bbM_{\V\U}}

\DeclareMathOperator{\rank}{rank}

\DeclareMathOperator{\Tr}{Tr}

\newcommand{\trp}{^{\mathsf{T}}}  
\newcommand{\ev}{\mathbb{E}} 

\newcommand{\C}{\mathbb{C}}
\newcommand{\R}{\mathbb{R}}

\newcommand{\N}{\mathbb{N}}

\newtheorem{theorem}{Theorem}[section]
\newtheorem{corollary}[theorem]{Corollary}
\newtheorem{proposition}[theorem]{Proposition}

\newtheorem*{definition}{Definition}

\begin{document}

\title{eGAD! double descent is \underline{e}xplained by \underline{G}eneralized \underline{A}liasing \underline{D}ecomposition    }

\author{Mark K.~Transtrum}
\email{mktranstrum@byu.edu}
\author{Gus L.~W.~Hart}
\affiliation{Department of Physics and Astronomy, Brigham Young University, Provo, Utah 84602, USA}
\author{Tyler J.~Jarvis}
\author{Jared P.~Whitehead}
\affiliation{Department of Mathematics, Brigham Young University, Provo, Utah 84602, USA}


\begin{abstract}

A central problem in data science is to use potentially noisy samples of an unknown function to predict function values for unseen inputs.
In classical statistics, the predictive error is understood as a trade-off between the bias and the variance that balances model simplicity with its ability to fit complex functions.
However, over-parameterized models exhibit counterintuitive behaviors, such as ``double descent'' in which models of increasing complexity exhibit \emph{decreasing} generalization error.  
Other models may exhibit more complicated patterns of predictive error with multiple peaks and valleys. 
Neither double descent nor multiple descent phenomena are well explained by the bias--variance decomposition.

We introduce a novel decomposition that we call the \emph{generalized aliasing decomposition} (GAD) to explain the relationship between predictive performance and model complexity.
The GAD decomposes the predictive error into three parts: 1.) \emph{model insufficiency}, which dominates when the number of parameters is much smaller than the number of data points, 2.) \emph{data insufficiency}, which dominates when the number of parameters is much greater than the number of data points, and 3.) \emph{generalized aliasing}, which dominates between these two extremes.

We demonstrate the applicability of the GAD to diverse applications, including random feature models from machine learning, Fourier transforms from signal processing, solution methods for differential equations, and predictive formation enthalpy in materials discovery. 
Because key components of the generalized aliasing decomposition can be explicitly calculated from the relationship between model class and samples without seeing any data labels, it can answer questions related to experimental design and model selection \emph{before} collecting data or performing experiments.
We further demonstrate this approach on several examples and discuss implications for predictive modeling and data science.
\end{abstract}

\maketitle

\section{Introduction}
\label{sec:introduction}

Predictive models allow scientists and engineers to extend data and anticipate outcomes for unseen cases.  A key issue for these models is the problem of how to understand and minimize the generalization error.  Traditionally, scientists think about generalization error in terms of a trade-off between bias and variance, but that trade-off does not readily predict the error curves for many models, especially models with more parameters than data points and models involving highly structured scientific and engineering data.   In this work, we introduce a new decomposition, the \emph{generalized aliasing decomposition (GAD)}, that explains a wide variety of error curves in predictive models for both small (classical) models and for large, over-parametrized models.  This decomposition explains complex generalization curves, including double and multiple descent, and can be used to inform the choice of model and experimental design (training points) to control, reduce, or even minimize generalization error.

Some of the fundamental choices when model building are (1) the sample data and (2) the complexity of the model class.
Simple models are generally preferred for many reasons, including interpretability and computational expense 
\cite{goldenfeld1999simple,hoel2013quantifyin,crutchfield2014dreams,transtrum2015perspective,mattingly2018maximizing,chvykov2021causal,quinn2022information}, but one of the more pragmatic justifications for parsimony is a desire to balance over- and under-fitting as understood through the bias--variance decomposition.
Models with few parameters avoid making wild predictions but under fit the observed data without much fidelity (high bias), while over-parameterized models fit the sampled data well with wild swings in between data points (high variance).
The unquestioned goal has been to find the ``sweet spot'' of model complexity that balances bias and variance, i.e., a faithful model of moderate complexity (see Figure~\ref{fig:Figure 1}, left panel) that minimizes the so-called ``risk'', that is, errors made by model predictions on unseen data.

\begin{figure*}
\begin{center}
    \includegraphics[width=\textwidth]{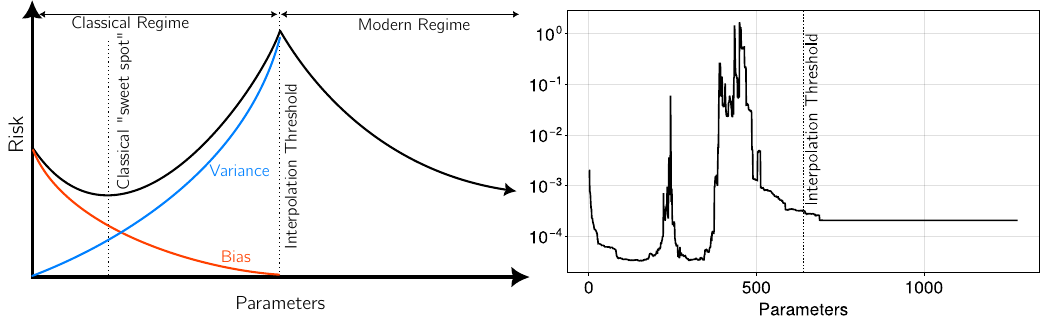}
\end{center}
\caption{\textbf{Limitations of the Bias--Variance Trade-off.}
  Left: Bias and variance are traditionally understood as monotonically decreasing/increasing contributions to risk to be balanced by tuning model complexity.
  Double descent illustrates a breakdown of this intuition beyond the interpolation threshold where variance and bias can exhibit counter-intuitive dependence on model class.
  Right: Highly structured data such as a cluster expansion of alloy formation enthalpy exhibit even more complicated and counter-intuitive dependence on model complexity, not easily explained using the bias--variance trade-off. 
      }
\label{fig:Figure 1}
\end{figure*}

While the foregoing story has long been the 
standard way to approach these problems, we now know this view of the fitting problem is not the whole story.
For extremely over-parameterized models (i.e., more parameters than samples), prediction errors may actually \emph{decrease} with additional parameters, a phenomenon often called ``double descent'' \cite{BelkinOriginal}, summarized by the left panel in Figure~\ref{fig:Figure 1}.
The boundary between the two regimes, where there are as many parameters as data points, is known as the \emph{interpolation threshold}, because it is (generically)  the boundary of where the model can perfectly interpolate the training data, but below that threshold interpolation cannot occur.
Non-convex risk curves (such as with double and multiple descent) are most famously recognized in neural networks \cite{YangYuYouSteinhardtMa2020,DarMuthukumarBaraniuk2021}, though this behavior has been observed in other settings as well.
(See \cite{GemanBienenstockDoursat1992} for the bias--variance decomposition for neural networks, \cite{HastieSurprises} for ordinary least-squares regression, and \cite{LoogVieringMeyKrijtheTax2020} for a thorough review.)
Furthermore, models and data sets can be designed to exhibit complex, multiple descent  \cite{BelkinMultiple, NakkiranKaplunBansalYangBarakSutskever2021,dAscoliRefinettiBiroliKrzakala2020,AdlamPennington2020,LeeCherkassky2022,OnetoRidellaAnguita2022,schaeffer2023double,LafonThomas2024}.

Models of highly structured data from scientific and engineering applications often exhibit similar multiple descent behavior.
For example, the risk curve in the right panel of Figure~\ref{fig:Figure 1} comes from a cluster expansion model of the formation enthalpy of alloy structures (see section \ref{sec:realworld} for more detail) in materials science.
Not only do the peaks and troughs appear \emph{to the left} of the interpolation threshold where classical bias--variance arguments ought to apply, the na\"ive interpolation threshold apparently plays no role.

Traditional data analysis techniques are also at odds with the intuition of the bias--variance decomposition.
The discrete Fourier transform, for example, is formally equivalent to a regression problem (see section~\ref{sec:Fourier}) with as many parameters (Fourier coefficients) as data, so
bias--variance arguments suggest that the inferred Fourier coefficients should exhibit unreasonable sensitivity to noisy data.
In spite of this, the fast Fourier transform which efficiently computes the discrete Fourier transform precisely at the interpolation threshold, is one of the most influential and widely used algorithms in all of science and engineering (even for noisy signals that are not band-limited).
Furthermore, techniques such as pseudospectral and collocation methods for solving differential equations are similarly equivalent to regression problems (see Section~\ref{sec:collocation}) but are known to exhibit optimal performance at or beyond the interpolation threshold \cite{boyd2001chebyshev}.

While the bias--variance decomposition holds as a formal mathematical result, these examples expose the limited insight it provides.
Its utility derives from the incorrect expectation that model selection balances the trade-off between monotonically decreasing (bias) and monotonically increasing (variance) error contributions.
In reality, the contributions of bias and variance for each of the preceding examples are non-monotonic, complex, and intimately connected with the algorithmic solution to the optimization problem.

Recent work has begun to explain these behaviors, often focusing on regression and the simplest case of double descent, although risk curves may be far more complicated \cite{BelkinMultiple}.
In \cite{AdlamPennington2020,schaeffer2023double}, the bias--variance decomposition is expanded to explain this non-convex behavior, relying on the interplay between the model design and the actual data.
Several other efforts have been made to clarify the relationship between the model class, inherent algorithmic bias, the split between testing and training data, and the appearance of non-monotonic loss and generalization curves.
We do not present a comprehensive summary of these efforts, but direct the interested reader to \cite{neal2018modern,dAscoliRefinettiBiroliKrzakala2020,NakkiranKaplunBansalYangBarakSutskever2021,LafonThomas2024}, 
which clarify the nature of double descent and its apparent reliance on the structure of the testing and training data sets.

In contrast to these approaches, we build on insights from signal processing \cite{DarMuthukumarBaraniuk2021}
and introduce a new decomposition (Eq.~\eqref{eq:GAD}), which we refer to as the \emph{generalized aliasing decomposition} (GAD),
summarized for the generic case of double descent in the left panel of Figure~\ref{fig:AliasingDecomposition}.
The aliasing decomposition explains generic risk curves in both the classical and modern regimes as the contribution of three terms: 1.) model insufficiency, 2.) data insufficiency, and 3.) generalized aliasing.

\begin{figure*}
  \includegraphics[width=\textwidth]{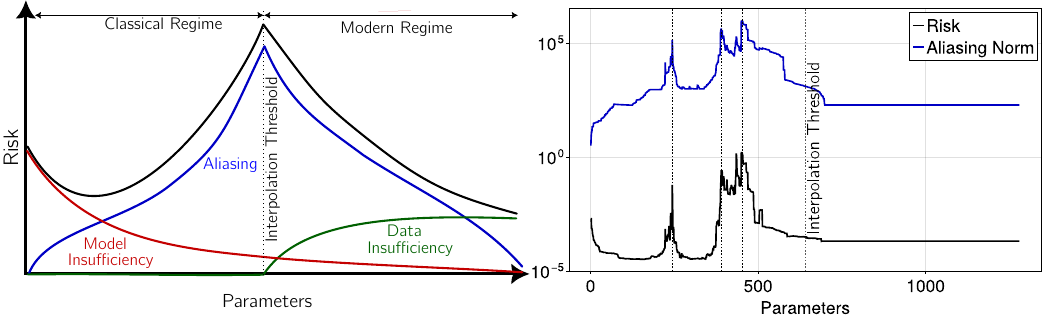}
  \caption{
    \textbf{Generalized Aliasing Decomposition.}
    Left: The generalized aliasing decomposition (GAD) expresses risk as the contribution of three terms: model insufficiency, data insufficiency, and generalized aliasing.
    Model insufficiency dominates for small models and decreases monotonically with the number of parameters.
    Data insufficiency dominates for large models, increasing monotonically with the number of parameters.
    Generalized Aliasing accounts for non-convex intermediate behaviors but has a single peak at the interpolation threshold for the generic case, accounting for the phenomenon of Double Descent.
    Right: For highly structured problems (see Figure~\ref{fig:Figure 1}, right), aliasing explains all of the non-convex behavior of generic risk curves at intermediate model sizes.    
      }
\label{fig:AliasingDecomposition}
\end{figure*}

\emph{Model insufficiency} quantifies the inability of the model to fit the data (red curve in the left panel of Figure~\ref{fig:AliasingDecomposition}).
It is usually the dominant error contribution for models with few parameters, and it decreases monotonically with the number of parameters.
Though the mapping is not exact, it roughly corresponds to ``bias'' in the bias--variance paradigm.
Adding more parameters to a model does not limit the ability of the model to fit data, so it decreases monotonically as we prove in Section~\ref{sec:model_insufficiency}.

\emph{Data insufficiency} quantifies how much model parameters are unconstrained by available data (green curve in the left panel of Figure~\ref{fig:AliasingDecomposition}).
It is the dominant error contribution for models with many excess parameters, increasing monotonically as the model grows.
Intuitively, adding more parameters does not introduce any additional parametric constraints, and we show this contribution does not increase with additional parameters.
However, adding parameters generically also imposes additional data requirements, so data insufficiency generally increases with the number of parameters (see Section~\ref{sec:data_insufficiency}).

Finally, \emph{generalized aliasing} explains all non-monotonic behavior in the intermediate regime (blue curve in the left panel of Figure~\ref{fig:AliasingDecomposition}).
The name derives from the special case of Fourier aliasing.  
When high-frequency (noisy) components of a signal cannot be distinguished from low-frequency components at finite sampling rates,  high-frequency (unmodeled) contributions are said to \emph{alias} with the low-frequency (modeled) components, corrupting the representation.

In the generic case, aliasing errors are maximized at the interpolation threshold and cause double descent (Figure~\ref{fig:AliasingDecomposition} left).
In structured cases such as the cluster expansion example of Section~\ref{sec:realworld}, aliasing also fully accounts for the complicated, non-monotonic behavior throughout both the classical and modern regimes (Figure~\ref{fig:AliasingDecomposition}, right).

As we demonstrate below, the GAD provides the intuition behind best practices for other analysis techniques, such as pseudospectral approaches to solving differential equations.
Although the contribution of generalized aliasing is non-monotonic in the number of parameters, we show its behavior is easily intuited.
In Section~\ref{sec:realworld}, we use this fact to explain the complicated risk curve in the right panel of Figure~\ref{fig:Figure 1}.

Taken collectively, \emph{the three components of the GAD explain all the qualitative features of generic risk curves}.
The GAD further clarifies the roles of model structure, data sampling, data labels, and the learning algorithm.  
Indeed, a useful feature of the decomposition is that much of the analysis can be done independently of data labels. 
Consequently, the GAD facilitates key modeling decisions such as the choice of model class, experimental design, regularization, and learning algorithm.

\section{The Generalized Aliasing Decomposition (GAD)}
\label{sec:derivation}

In this section we give the details of the GAD and its mathematical justification.  
We begin with establishing notation to describe the regression problem, define the decomposition, and then describe how the decomposition influences the error or risk of the fitting problem.  Several explicit examples and applications are given in Section~\ref{sec:applications}.

 Readers wishing to focus on the examples and discussion should first read Sections~\ref{sec:aliasing} and \ref{sec:results} where the GAD is defined, but can safely skip over Sections~\ref{sec:GAD_overview}--\ref{sec:overmodeling} and all but the first paragraph of Section~ \ref{sec:data_insufficiency}.

\subsection{Mathematical Preliminaries}
\label{sec:aliasing}

In regression, data $\y$ are given at samples of an independent variable $t$ and usually decomposed as the sum of an unknown signal $f_{\btheta}(t)$ parameterized by $\btheta$ and noise $\xi$:
\begin{align}
  \label{eq:regression}
  y_i & = f_{\btheta}(t_i) + \xi_i,
\end{align}
where the subscript $i$ refers to a particular data sample.
In standard statistical practice, one next chooses a functional form for the model $f_{\btheta} (t)$ and and ansatz for the distribution of the noise $\xi_i$.
In regression, by far the most common assumption is that the noise terms are Gaussian distributed which leads to a least squares regression.
For linear regression the signal is a linear combination of basis functions $f(t) = \sum_j \Phi_j(t) \theta_j$, so that the fundamental regression equation \eqref{eq:regression} becomes
\begin{align}
  \label{eq:linearregression}
  \y & = \bbM \btheta + \bxi,
\end{align}
where the design matrix $\bbM$ is composed of samples of basis functions, $\bbM_{ij} = \Phi_j(t_i)$.

As an illustration consider a polynomial fit on an interval $[a,b]$, and
take the basis functions to be the usual monomial basis $\{1, t, t^2, t^3, \dots, t^d\}$ for some $d>0$ \footnote{Although fitting monomials is the canonical pedagogical example, the ill conditioning of this Vandermonde matrix makes this basis ill-suited for practical applications.}; so $\Phi_j = t^{j-1}$, and the design matrix $\bbM$ is 
\begin{align}
    \label{eq:Vandermonde}
\bbM = \begin{pmatrix}
1 & t_1 & t_1^2 & \dots & t_1^d\\
1 & t_2 & t_2^2 & \dots & t_2^d\\
\vdots & & & & \vdots\\
1 & t_n & t_n^2 & \dots & t_n^d
    \end{pmatrix}.
\end{align}

Inferred parameter values $\hat{\btheta}$ are found by inverting the design matrix.
Since $\bbM$ is generally not square, an appropriate pseudoinverse $\bbM^+$ is used:  $ \hat{\btheta} = \bbM^+ \y$.
The Moore--Penrose pseudoinverse is the standard choice, corresponding to the least squares loss, for linear regression\cite{HumpherysJarvisEvans2017}, including in this motivating example.
Other cases may require an algorithmic solution, but common algorithmic choices, such as stochastic gradient descent, are known to produce similar norm-minimizing solutions (see \cite{sheng2013iterative,gower2017randomized,schaeffer2023double}, for example).

Finally, predictions at unobserved values of the independent variable are constructed
\begin{align}\label{eq:estimated-regression-function}
  \hat{y}(t) & = \sum_j \Phi_j(t) \hat{\btheta}_j.
\end{align}
An important quantity of interest for validation is the squared error, averaged over a (typically theoretical) distribution of all of the data, not just the training samples. The expectation of the squared error is called \emph{generalization error} or \emph{population risk}  
\begin{equation}\label{eq:pop_risk}
    R_{\btheta}(\hat\y) =  \ev [(\y - \hat\y)^2],
\end{equation} where the dependence of $R_{\btheta}$ on the model class and training data  is implicit.  A primary goal in data science is to identify the model class and degree of complexity that minimizes this risk  \eqref{eq:pop_risk}.

\subsection{Generalized Aliasing}\label{sec:results}

With a common vocabulary established, we now derive the aliasing operator that underpins the generalized aliasing decomposition (GAD).
We no longer require that the data points $t$ lie in $\R$; they could belong to any set $\Omega$.  But we assume that the model functions $\Phi_j:\Omega \to \R$ may be extended to form a complete set, meaning that the true function $y(t)$ (both signal $f$  and noise $\xi$) can be uniquely expressed as a convergent series 
\begin{align}
  y(t) &= \sum_j \Phi_j(t) \theta_j.
\end{align}
 on the entire domain, not just on the training points.  
For the example of polynomial regression, if $y(t)$ is a real analytic  function, then the infinite monomial basis $\{1,t,t^2,\dots\}$ is complete and an appropriate extension for this example.
Said more formally, the noisy signal $y(t)$ is an abstract vector $\y$ in a (potentially infinite-dimensional) vector space $\D$ expressed in some Schauder basis $\{\Phi_j\}_{j\in \N}$ as
\begin{align}\label{eq:define-M}
  \y & = \bbM \btheta,
\end{align}
where $\bbM$ is a bounded linear transformation mapping the vector $\btheta$ in the parameter space $\Theta$ to $\y$ in the data space $\D$. 
In the case of fitting a polynomial on an interval $[a,b]$, the operator $\bbM$ could be thought of as a generalized Vandermonde matrix with countably (infinite) many rows corresponding to rational points of $[a,b]$ and countably (infinite) many columns corresponding to $t^j$ for each nonnegative integer $j$.\footnote{Any real analytic function is determined by its values on a dense set, so we may limit ourselves to only considering rational points $t$ in the interval $[a,b]$.}

Performing linear regression on samples of $y(t)$ and making predictions at unobserved values of $t$ corresponds to partitioning data space $\D$ into a direct sum $\T \oplus \V$ of training $\T$ and prediction $\V$ subspaces. We write $\by$ in this decomposition as $\by = (\by_\T, \by_\V)$.  We assume that $\T$ has finite dimension $n$, but $\V$ need not be finite dimensional. The learning problem is this: Given observations in $\by_\T$, predict the components of $\by_\V$.

In practice, this is done by similarly partitioning the representation space $\Theta = \M \oplus \U$ into a modeled $\M$ and an unmodeled $\U$ subspace  so that $\btheta = (\btheta_\M, \btheta_\U)$, implicitly assuming that $\btheta_\U$ are negligible.  We usually assume that $\M$ has finite dimension $m$ (we have $m = d+1$ for polynomials of degree at most $d$), but $\U$ need not be finite dimensional. 
With these partitions, the relationship of Eq.~\eqref{eq:define-M} between data and coordinates takes the block representation described in the definition below.
\begin{definition}
Denote the decomposition of the labeled data as
\begin{align}
  \label{eq:partition}
  \left(
  \begin{array}{c}
    \by_\T \\
    \by_\V
  \end{array}
  \right) & =
            \left(
            \begin{array}{cc}
              \MTM & \MTU \\
              \MVM & \MVU \\
            \end{array}
            \right)            
  \left(
    \begin{array}{c}
    \btheta_\M \\
    \btheta_\U
  \end{array}
  \right),
\end{align}
where the linear transformation $\MTM:\R^m \to \R^n$ is the usual \emph{design matrix}.
\end{definition}

By explicitly recognizing the unmodeled blocks in the definition, we emphasize that some contributions to the signal $y(t)$ will remain unknown to us (noise, for example).
Essentially, we acknowledge that our model is a subspace of a universal function space which will in turn allow us to reason about the relationship between the modeled and the unmodeled spaces.  
The significance of this decomposition, as opposed to simply treating unmodeled signal as noise, is discussed further in section~\ref{sec:nescience}.
We emphasize that with this definition $\bbM$ does not denote the classical design matrix.  Rather the block $\MTM$ is the classical design matrix, and  $\bbM$  represents a full basis transformation (with both modeled and unmodeled basis functions) on the complete signal (including both seen and unseen data).
Because $\bbM$ is bounded and linear, the subblocks $\MTU$, $\MVM$, and $\MVU$ are also bounded linear transformations.

In the case of fitting a polynomial of degree at most $d$ on $n$ training points $t_1,\dots t_n$, the design matrix (upper left block) $\MTM$ is the Vandermonde matrix in \eqref{eq:Vandermonde}
and the unmodeled training (upper right) block is the semi-infinite matrix 
\[
\MTU = 
\begin{pmatrix}
t_1^{d+1} & t_1^{d+2} & \dots \\
t_2^{d+1} & t_2^{d+2} & \dots \\
  \vdots &  \vdots \\
t_n^{d+1} & t_n^{d+2} & \dots \\
\end{pmatrix}.
\]
The rows of the lower blocks $\MVM$ and $\MVU$ correspond to the prediction points,  $t_\V= [a,b]\setminus \{t_1,\dots, t_n\}$ (again, a countable dense subset of points in $[a,b]\setminus \{t_1,\dots, t_n\}$ suffices).  The columns of the lower left block $\MVM$ correspond to the monomials $1,t,t^2,\dots, t^d$ (spanning the space $\M$) evaluated at the points $t_\V$. 
The columns of the lower right block $\MVU$ correspond to the unmodeled monomials $t^{d+1}, t^{d+2}, \dots$ (spanning $\U$), evaluated at points $t_\V$. 

We learn the modeled parameters $\hat{\btheta}_\M$ using some pseudoinverse $\MTM^+$ of the design matrix $\MTM$:
\begin{align}
  \label{eq:thetahat}
  \hat{\btheta}_\M & = \MTM^+ \,\by_\T.
\end{align}
Inferring only $\hat{\btheta}_\M$ is equivalent to assuming that the unmodeled parameters vanish, so $\hat{\btheta}_\U = \0$ and $\hat{\btheta} = (\hat{\btheta}_\M, \0)$. 
However, the true representation of the training data $\by_{\T}$ includes contributions from both the modeled and unmodeled components of $\btheta$:
\begin{align}
\label{eq:regression_nescience}
\by_\T = \MTM \btheta_\M + \MTU \btheta_\U.
\end{align}
The unmodeled term $\MTU \btheta_\U$ corresponds to the noise in Eq.~\eqref{eq:regression}.
Rather than assume a particular distribution for the noise as one does in standard statistical practice, we leave the unmodeled term arbitrary and study the sensitivity of inference to the presence of unmodeled noise.
The inferred parameters $\hat{\btheta}_\M$ are distorted by the unmodeled term, which, in our extended representation, takes the form:
\begin{align}
  \label{eq:thetahatvstrue1}
  \hat{\btheta}_\M & = \left(\MTM^+ \MTM\right) \btheta_\M + \left(\MTM^+ \MTU \right) \btheta_\U.
\end{align}
For conceptual clarity, we write this as 
\begin{align}
  \hat{\btheta} = \left(
            \begin{array}{c}
              \hat{\btheta}_\M \notag\\
              0
            \end{array}
            \right)
            & = \left(
            \begin{array}{cc}
              \MTM^+ \MTM & \MTM^+ \MTU \\
              0 & 0
            \end{array}
            \right)
            \btheta \\
          & = \left(
    \begin{array}{cc}
      \bbB & {\bbA} \\
              0 & 0
            \end{array}
            \right)
    \btheta,
\end{align}
where we have defined
\begin{equation}\label{eq:bbA}
    \bbA  =  \MTM^+ \MTU \quad\text{ and }\quad \bbB = \MTM^+ \MTM,
\end{equation}
and $\btheta$ is the vector of parameters that represents the complete signal precisely.

We call ${\bbA}$ the \emph{generalized aliasing operator}.  It quantifies how the effects of the unmodeled parameters $\btheta_\U$ are redirected (aliased) into the modeled parameters.
Note that ${\bbA}$ depends not only on the partition between modeled parameters and unmodeled modes, but also on the partition between training points and prediction points and the choice of pseudoinverse or the choice of learning algorithm, more generally.

In the familiar example of Fourier series, the concept of \emph{aliasing} refers to the distortion of a low-frequency signal by high-frequency modes.
Expressed in the form we have described, Fourier aliasing is found from $\bbA = \MTM^+ \MTU$, expressed in the Fourier basis for uniform samples (see Section~\ref{sec:Fourier} for an example of aliasing in Fourier series), where it can be expressed in closed-form.  
Generalizing beyond the specific concept of frequency, $\bbA$ quantifies how unmodeled components affect the signal at the sampled points, leading to a misrepresentation of the inferred modeled parameters that we call \emph{generalized aliasing}.

Using $\hat{\btheta} = (\hat{\btheta}_\M, \0)$, we reconstruct the inferred signal over both training and prediction points
\begin{align}
  \label{eq:predictions}
  \hat{\by} & = \bbM \hat{\btheta}  = \bbM \MTM^+ \by_\T 
  = \bbM ( \bbB \btheta_\M  + \bbA \btheta_\U).
\end{align}
Comparing  $\hat{\by}$ with the true 
$  \by $,
the GAD decomposes the population risk of Eq.~\eqref{eq:pop_risk} into an intuitive partition.
Assuming that the points $t$ are drawn from a uniform distribution on $\Omega$, the risk is
\begin{align}
R_{\btheta}(\hat{\by}) &=  \ev_t [(\by(t) - \hat{\by}(t))^2]
= \sum_{t} (\by(t) - \hat{\by}(t))^2\notag\\
& = \|\by - \hat{\by}\|^2  \label{eq:ey}  \\
& = \left \| \bbM
    \begin{pmatrix}
           {{\bbI_{\M} - \bbB}} & -\bbA \\
            0 &  \bbIU
    \end{pmatrix}
     \btheta \right\|^2, \label{eq:GAD-risk}
\end{align}
where the norm $\|\cdot\|$ is the $2$-norm $\|\cdot\|_2$, and ${\bbI_{\M}}$ and $\bbIU$ are the identity operators on $\M$ and $\U$, respectively.
This  motivates the definition of the \emph{parameter error operator}
\begin{align}
  \bbE_{\btheta} = \left(
          \begin{array}{cc}
            \bbPN & {-\bbA} \\
            0 & \bbIU
          \end{array}
  \right),
  \label{eq:GAD}
\end{align}
where we define
\[
\bbPN = \bbI_\M - \bbB.
\]
We have used the subscript $\btheta$  on $\bbE_{\btheta}$ to indicate that $\bbE_{\btheta} \btheta = \btheta - \hat{\btheta}$ represents errors in the inferred parameters; whereas the errors in the signal are $\by-\hat{\by} = \bbM \bbE_{\btheta} \btheta$.  The notation $\bbPN = \bbI_\M - \bbB$ is motivated  by the fact that it is the orthogonal projection onto the kernel $\K$ of $\MTM$ (see Proposition~\ref{prop:DI-is-projection}, below).

We call the block decomposition in Eq.~\eqref{eq:GAD} the \emph{generalized aliasing decomposition}, or \emph{GAD}, for short.
We call the three nonzero blocks of $\bbE_{\btheta}$ \emph{data insufficiency} $\bbPN$, \emph{model insufficiency} $\bbIU$, and \emph{generalized aliasing} $-\bbA$.  The effect on the signal of these operators acting on the parameters $\btheta$ are depicted in the left panel of Figure~\ref{fig:AliasingDecomposition}.
These terms also have a geometric interpretation in the parameter space illustrated in Figure~\ref{fig:GADGeometry}.
We now analyze each of these contributions to the error in turn.

\begin{figure}
    \centering
    \includegraphics[width=\columnwidth]{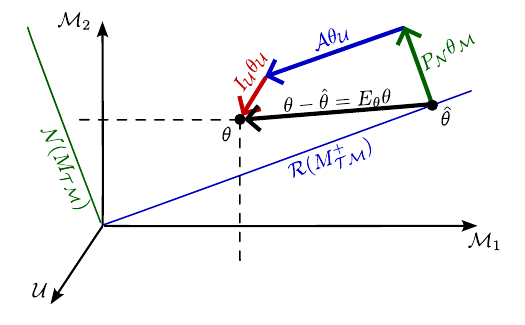}
    \caption{\textbf{Geometry of the GAD}.
      The three terms of the GAD decompose the error $\bbE_{\btheta} \btheta = \btheta - \hat{\btheta}$ in the inferred parameters  into three orthogonal components.
      Working backwards from the true parameters $\theta$ to the inferred parameters $\hat{\theta}$, (subtracting off) the model insufficency $\bbIU \theta_\U$ projects the full $\theta$ into the modeled subspace (two dimensional in this figure, corresponding to the axes $\M_1$ and $\M_2$).
      The aliasing $\bbA \theta_\U$  perturbs in the range $\mathcal{R}(\MTM^+)$ of $\MTM^+$.
      Finally, the data insufficiency $\bbPN \theta_M$ chooses the minimal norm solution by shifting through the kernel $\mathcal{N}(\MTM)$ of $\MTM$.
    } 
    \label{fig:GADGeometry}
\end{figure}

\subsection{Error Analysis}
\label{sec:error}

For a given partition of the parameter space between modeled  and unmodeled subspaces $\Theta = \M\oplus \U$, the predictions $\hat{\by}$  and the risk $R_{\btheta} (\hat{\by}))  = \|  \hat{\by} - \by\|^2$ depend on the choice $T = \{t_1, \dots, t_n\}$ of the training points.  
The expected value  
 of the risk (taken over the distribution of $T$) is often decomposed into a sum of a \emph{bias} term and a \emph{variance} term; see, for example, \cite[\S 20.1]{Wasserman}.
In many settings, however, it is more natural to analyze the risk $\|\by-\hat{\by}\|^2$ and its expected value $\ev_T\!\left[\|\by-\hat{\by}\|^2\right]$ directly through the GAD.  

In this Section we are primarily interested in estimating the risk for a particular decomposition.
Observe that the risk is bounded by the (square of the) product of three norms
 \[
R_{\btheta}(\hat{
\by}) = \|\bbM \bbE_{\btheta} \btheta\|^2 \le \|\bbM \|^2 \|\bbE_{\btheta}\|^2  \|\btheta\|^2.
 \]
 The norm used on the linear transformations $\bbM$ and $\bbE_{\btheta}$ is the \emph{induced} norm, defined as 
\[
\|\bbM\| = \max_{\|\bnu\|=1} 
\|\bbM \bnu\|
\qquad \text{and} \qquad
\|\bbE\| = \max_{\|\bnu\|=1} 
\|\bbE \bnu\|.
\]
In the finite-dimensional case (where the transformations are represented by matrices) it is well known that when the norm on both the domain and range of a matrix is the usual (two-) norm, then the induced norm of a matrix is its largest singular value. In this case the induced norm is often called the \emph{spectral norm}.  

We assume the transformation $\bbM$ has a bounded norm and note that its norm is independent of the choice of model $\M$ and of the choice $T$ of training points.  The risk certainly depends on $\btheta$ and its norm, particularly in the two extremes of low and high model complexity (left and right end, respectively, of the plots shown here).  In the intermediate regime however the risk is generally dominated by the aliasing component of the GAD which is mostly independent of the $\btheta$ themselves.

Since we are particularly interested in how these contributions depend on the number of model parameters, consider the situation where the columns of $\bbM$ are fixed, and a given decomposition  $\Theta = \M \oplus \U$ of parameter space is changed by moving one basis element $\Phi$ out of the space $\U$ and into the space $\M$.  This corresponds to moving the corresponding column $\bphi$ out of the matrix $\MTU$ and into the design matrix $\MTM$. Fixing the order of the columns of $\bbM$, and letting $m$ denote the dimension of the modeled space $\M$, we introduce the notation $\MTM(m)$ and $\MTU(m)$ to make explicit the dependence of the block operators on the dimension $m$. That is, $\MTM(m)$ denotes the design matrix in the case that the first $m$ columns of $\bbM$ are assigned to $\MTM$ and the remaining columns are assigned to $\MTU$.
With this convention, we now analyze each of the elements of the GAD in turn.

\subsubsection{Data Insufficiency \texorpdfstring{$\bbPN$}{PN}}\label{sec:data_insufficiency}

Data insufficiency refers to the upper left block $\bbPN$ of Eq.~\eqref{eq:GAD} and its effect on the parameters $\btheta$.
In the case that $\MTM^+$ is the Moore--Penrose pseudo inverse of $\MTM$, the operator $\bbPN$ is the orthogonal projection of $\M$ onto the null space $\K \subseteq \M$ of $\MTM$, as the following proposition shows.
\begin{proposition}\label{prop:DI-is-projection}
If $\MTM^+$ is the Moore--Penrose pseudoinverse of $\MTM$, then the data insufficiency operator $\bbPN = \bbI_\M - \bbB$ is the orthogonal projection of $\M$ to the kernel $\K$ of $\MTM$.
\end{proposition}
\begin{proof}
Let 
\[
\MTM = [U_1 | U_2] 
    \begin{bmatrix}
    \Sigma_1 & 0 \\
    0        & 0
    \end{bmatrix}
\begin{bmatrix}
    V_1\trp\\
    V_2\trp
\end{bmatrix}
\] be the full SVD of $\MTM$, where $\MTM$ has rank $r$, the matrix $\Sigma_1$ is invertible of shape $r\times r$, and $V = [V_1 | V_2]$ is an orthogonal matrix 
\[
V_1 V_1\trp + V_2 V_2\trp = V V\trp = \bbI_{\M},
\]
with $V_1$ having $r$ columns and the columns of $V_2$ spanning the kernel $\K$ of $\MTM$. 
The Moore--Penrose pseudoinverse of $\MTM$ can be written as 
$\MTM^+ = V_1 \Sigma_1^{-1} U_1\trp$, which gives  
\begin{align*}
\bbI_\M - \bbB & = V V\trp - \MTM^+ \MTM\\
& = (V_1 V_1\trp + V_2 V_2\trp) - V_1 \Sigma_1^{-1} \Sigma_1 V_1\trp\\
& = V_2 V_2\trp.
\end{align*}
But since the columns of $V_2$ span the kernel $\K$ of $\MTM$, the matrix $V_2 V_2\trp$ is exactly the orthogonal projection of $\M$ onto  $\K$.
\end{proof}
The induced norm of any projection operator is always either $0$, if it's the zero operator, or
$1$ otherwise, which yields the following corollary.
\begin{corollary}
    The induced norm $\|\bbPN\|$ of the operator $\bbPN$ is bounded above by $1$, and is always equal to $1$ except when $\MTM$ is injective (full column rank), in which case $\bbPN = 0$ is the zero operator.
\end{corollary}
The corollary shows that when there are enough  training data so that $\MTM$ is of full column rank, then the norm of the data insufficiency operator $\bbPN$ is zero. Generically this happens when there are more data points (rows) than basis functions (columns) in $\MTM$,  as depicted in the left panel of Figure~\ref{fig:AliasingDecomposition}.

Let $\K(m)$ denote the null space of $\MTM(m)$. Since $\K(m) \subseteq \K(m+1)$, it follows that the error contribution $\| \bbPN \btheta_\M\|$ from data insufficiency is a nondecreasing function of $m$. This is depicted on the bottom right of the left panel of Figure~\ref{fig:AliasingDecomposition}.

\subsubsection{Model Insufficiency  \texorpdfstring{$\bbIU$}{IU}}\label{sec:model_insufficiency}

\emph{Model insufficiency} refers to the lower right block $\bbIU$ of Eq.~\eqref{eq:GAD}.
It is the most straightforward of the three parts of the GAD to analyze, as it is simply the identity operator on the unmodeled parameters $\U$.
Except in the trivial and uninteresting case that $\dim (\U) = 0$, its operator norm is always $1$.  The contribution to the parameter error from model insufficiency is
simply the square $\| \btheta_\U \|^2$ of the norm of the nescient parameters.  It follows that model insufficiency is a non-increasing function of $m$ since it decreases by exactly $\vert \theta_{m+1} \vert^2$ as the coordinate $\theta_{m+1}$ is removed from the unmodeled space $\U$ and adjoined to the modeled space $\M$.

Model insufficiency dominates when the dimension $m$ of the model $\M$ is small, reflecting the fact that most of the signal is unknown and the model lacks the capacity to capture the signal faithfully (see the bottom left part of the left panel of Fig.~\ref{fig:AliasingDecomposition}).

\subsubsection{Generalized Aliasing  \texorpdfstring{$\bbA$}{A}: Overview}\label{sec:GAD_overview}

Finally, we consider contributions from the \emph{generalized aliasing operator} $\bbA$.
This is the most complicated contribution, and is the source of non-trivial generalization curves such as double or multiple descent, or multiple risk peaks from structured data.  This term tends to dominate the risk in the intermediate regime of most models.  Importantly, in many cases the effects of $\bbA$ can be analyzed without knowing $\btheta$ or the labels $\y$. 

Recall from Eq.~\eqref{eq:bbA} that the aliasing operator is the product of the pseudoinverse design matrix $\MTM^+$ and the transformation $\MTU$.
Increasing the number of model parameters by moving one column $\bphi$ out of $\MTU$ and into the design matrix never increases the 
norm $\|\MTU\|$, but its effect on 
$\|\MTM^+\|$ is determined primarily by whether $\bphi$ 
is linearly independent of the other columns of $\MTM$ or not, as described in the following theorem (proved in Section~\ref{sec:math:A}).
\begin{theorem}\label{thm:main}
When changing the model by moving one column $\bphi$ out of $\MTU$ and into the design matrix,  the norm 
$\|\MTU\|$ never increases and 
\begin{itemize}
\item  $\|\MTM^+\|$ cannot decrease if $\bphi$ is linearly \textbf{independent} of the other columns of $\MTM$, 
\item $\|\MTM^+\|$ cannot increase if $\bphi$ is linearly \textbf{dependent} upon the other columns of $\MTM$.
\end{itemize}
Moreover, as the model dimension $m$ increases to $\infty$, the norm $\|\MTM^+\|$ shrinks to $0$, almost surely.
\end{theorem}
Although it can be arranged so that $\|\MTM^+\|$  remains constant when moving one column $\bphi$ fom $\MTU$ to $\MTM$, in most cases we see a significant  increase in $\|\MTM^+\|$ whenever $\bphi$ is independent from the previous columns of $\MTM$ and a significant decrease in $\|\MTM^+\|$ whenever $\bphi$ is dependent upon the previous columns of $\MTM$.

Theorem~\ref{thm:main} fully explains the sharp peaks in generalization curves described as double and multiple descent, and it is relevant to other nonmonotonic features in both the under- and over-parameterized regimes, as we now describe.

For a generic $\bbM$ the columns are typically arranged so that for $m < n$ each column $\bphi_{m+1}$ is independent of the previous columns of $\MTM(m)$ and each column of $\MTU$ does not have a large impact on the norm of $\MTU$.
Hence, as $m$ increases the norm $\|\MTM^+(m)\|$ is expected to grow nearly monotonically until the interpolation threshold $m=n$.  Once $m\ge n$ the columns of $\MTM(m)$ are expected to span the column space of the entire training set $(\MTM | \MTU)$ of the operator $\bbM$, so each new column added to  $\MTM(m)$ will be linearly dependent on the existing columns, and hence the norms $\|\MTM^+(m)\|$ and $\|\bbA\|$ cannot increase and typically decrease.  
In this generic case, $\|\MTM^+(m)\|$ is a nondecreasing function of $m$ until $m=n$, after which it is nonincreasing.
The common peak in the generalization error at the interpolation threshold is thus understood as the peak in $\|\MTM^+(m)\|$ at $m = n$.

More complicated generalization curves can be understood by considering whether the next basis vector $\bphi_{m+1}$ is either linearly dependent (the norm $\|\MTM^+(m+1)\|\leq \|\MTM^+(m)\|$) or linearly independent (the norm $\|\MTM^+(m+1)\| \geq \|\MTM^+(m)\|$) on the previously modeled terms, i.e., all those columns already contained in $\MTM(m)$. Regardless of the ordering of the columns of $\bbM$, the upper bound
\[
\|\bbA\| \le \|\MTM^+\| \|\MTU\|
\]
cannot increase when stepping from $m$ to  $m+1$ unless the next column $\bphi_{m+1}$ is independent of the previous columns.  Moreover, this upper bound will almost surely shrink to $0$ as $m\to \infty$ as both $\|\MTU\| \rightarrow 0$ and $\|\MTM^+\|\rightarrow 0$.

Of course, one can arrange to add columns to $\MTM(m)$ in a way that the rank of $\MTM(m)$ grows slower than expected, permitting the construction of descent curves for $\|\bbA\|$ of various shapes.  But when the columns are sufficiently general (as, for example, with the random ReLU features (RRF) model and the random Fourier features (RFF) model), the result for $\|\bbA\|$ is similar in shape to the standard double descent curve for mean-squared error, described in \cite{BelkinOriginal} with a single large peak at the interpolation threshold and decreasing monotonically thereafter (see Figure~\ref{fig:belkin-MM-plot}). 
 
\subsubsection{Generalized Aliasing  \texorpdfstring{$\bbA$}{A}: Mathematical Treatment}\label{sec:math:A}

In this section we give more mathematical details of the norm $\|\bbA\|$ of the aliasing operator and a proof of Theorem~\ref{thm:main}.

\paragraph{Tools for Analyzing Norms}
The main tool we use is the following theorem, whose earliest statement seems to be \cite[Theorem 17]{Gantmacher} (see also \cite{BunchNielsen, Thompson, Wilkinson}).
\begin{theorem}\label{thm:interlaced-eigenvalues}
Let $\Phi$ be an $n\times n$ Hermitian matrix with eigenvalues $\alpha_1\ge \alpha_2 \ge \cdots \ge \alpha_n$ and let $C$ be a positive semidefinite matrix of rank $1$.  The eigenvalues  $\beta_1\ge \beta_2 \ge \cdots \ge \beta_n$ of the matrix $\Xi = \Phi + C$ satisfy 
\[
\beta_1 \ge \alpha_1 \ge \beta_2 \ge \alpha_2 \ge \cdots \ge \beta_n \ge \alpha_n.
\]
\end{theorem}

This theorem immediately gives the corollary that, under the same assumptions on $\Phi$ and $C$, the eigenvalues $\delta_1 \ge \cdots \ge \delta_n$ of $D = \Phi - C$ are below the corresponding eigenvalues of $\Phi$ and are interleaved according to:
\[
\alpha_1 \ge \delta_1 \ge \alpha_2 \ge \cdots  \ge \alpha_n \ge \delta_n.
\]

Theorem~\ref{thm:interlaced-eigenvalues} also leads to the following fundamental result for analyzing the operator norm of the aliasing operator $\bbA$, stated here in a more general form. 
\begin{theorem}\label{thm:add-column-inverse-norm}
Let $\bbX$ be an $n\times m$ matrix of rank $r$ with smallest  singluar value $\sigma_r>0$.  Let $\widetilde{\bbX} = [\bbX | \bphi]$ be the $n\times (m+1)$ matrix obtained by adjoining an $n$-dimensional column vector $\bphi$ to $\bbX$.  The smallest singular value $\widetilde{\sigma}_{\mathrm{min}}>0$ of $\widetilde{\bbX}$ satisfies the following relations:
\begin{align*}
0 & < \widetilde{\sigma}_{\mathrm{min}} \le \sigma_r  & 
\text{if $\rank(\bbX) < \rank(\widetilde{\bbX})$,}\\
0 & < {\sigma}_r \le \widetilde{\sigma}_{\mathrm{min}}  & \text{if $\rank(\bbX) = \rank(\widetilde{\bbX})$.}
\end{align*}
\end{theorem}
\begin{proof}
Both $\bbX \bbX\trp$ and $\widetilde{\bbX}\widetilde{\bbX}\trp$ are $n \times n$ positive definite Hermitian matrices.
The singular value decomposition of $\bbX$ shows that the singular values $\sigma_1\ge \cdots \ge \sigma_r >0$ and the eigenvalues $\lambda_1\ge \cdots \ge \lambda_m$ of $\bbX \bbX\trp$ satisfy 
\[
\lambda_1 =\sigma^2_1 \ge \lambda_2 = \sigma^2_1 \ge \cdots \ge \lambda_r = \sigma^2_r > 0 = \lambda_{r+1}.  
\]
Similarly, the singular values $\widetilde{\sigma}_1 \ge \widetilde{\sigma}_2 \ge \cdots$ and eigenvalues $\widetilde{\lambda}_1 \ge \cdots \ge \widetilde{\lambda}_m$ of $\widetilde{\bbX} \widetilde{\bbX}\trp$ satisfy
\[
\widetilde{\lambda}_1 = \widetilde{\sigma}^2_1 \ge \widetilde{\lambda}_2 = \widetilde{\sigma}^2_2 \cdots \ge \widetilde{\lambda}_r = \widetilde{\sigma}^2_r \ge  \widetilde{\lambda}_{r+1} \ge \cdots,
\]
where $\widetilde{\lambda}_{r+1} = 0$ if $\rank(\widetilde{\bbX}) = r$, but $\widetilde{\lambda}_{r+1} > 0$  if $\rank(\widetilde{\bbX}) = r+1$.
Expanding $\widetilde{\bbX}\widetilde{\bbX}\trp$ gives
$\widetilde{\bbX}\widetilde{\bbX}\trp = \bbX \bbX\trp + \bphi\bphi\trp$,
where $\bphi \bphi\trp$ is positive semidefinite, so Theorem~\ref{thm:interlaced-eigenvalues} implies that 
\[
\widetilde{\lambda}_1 \ge \lambda_1 \cdots \ge \lambda_{r-1} \ge \widetilde{\lambda}_r \ge \lambda_r \ge \widetilde{\lambda}_{r+1} \ge 0. 
\]

If $\rank(\widetilde{\bbX}) = r+1$ (that is, $\bphi$ is not in the column space of $\bbX$), then 
$
\lambda_r = \sigma^2_r \ge \widetilde{\lambda}_{r+1} = \widetilde{\sigma}^2_{r+1} >0$.
Taking square roots gives  
$\sigma_r > \widetilde{\sigma}_{r+1} = \widetilde{\sigma}_{\mathrm{min}} > 0$.

If $\rank(\widetilde{\bbX}) = r$ (that is, $\bphi$ is in the column space of $\bbX$), then the smallest nonzero eigenvalue of  $\widetilde{\bbX}\widetilde{\bbX}\trp$ is $\widetilde{\lambda}_r$, which satisfies  
$\lambda_{r-1} \ge \widetilde{\lambda}_r \ge \lambda_r > 0$. 
Taking square roots gives 
$\widetilde{\sigma}_{\mathrm{min}}=\widetilde{\sigma}_{r}  \ge \sigma_r > 0$, as required.
\end{proof}

\paragraph{Decomposing Generalized Aliasing}
We are interested in how the (induced) operator norm 
\[
\|\bbA\| = \|\MTM^+ \MTU\| \le \|\MTM^+\| \|\MTU\|
\]
changes as the model grows, that is, as a new column is removed from $\MTU$ and added to $\MTM$, but the training set (which rows are included) remains unchanged.

As before, assume $\bbM$ is fixed, and $\MTM(m)$ corresponds to the design matrix block when the model consists of the first $m$ columns of $\bbM$, and $\MTU(m)$ is the corresponding unmodeled block.
The matrix $\MTM(m+1)$ is constructed by moving one column $\bphi_{m+1}$ from the nescience block into the design block.  The  training-set (top) part of the operator $\bbM$ decomposes as    
$\begin{bmatrix}\MTM(m) & \bphi_{m+1} & \MTU(m+1)
\end{bmatrix}$.

\paragraph{Norm of $\MTU$}
First consider what happens to the matrix $\MTU$ when a column $\bphi_{m+1}$ is removed from $\MTU(m) = [\bphi_{m} \quad \MTU(m+1)]$.
Expanding the product $\MTU(m) \MTU(m)\trp$ gives 
$\MTU(m)\MTU(m)\trp = \bphi_{m+1}\bphi_{m+1}\trp + \MTU(m+1)\MTU(m+1)\trp$.  
Since $\bphi_{m+1}\bphi_{m+1}\trp$ is positive semidefinite,  Theorem~\ref{thm:interlaced-eigenvalues} applies and 
guarantees that the 
norms satisfy $\|\MTU(m)\| \ge \|\MTU(m+1)\|$, and thus the norm $\|\MTU(m)\|$ is a nonincreasing function of $m$.  

\paragraph{Pseudoinverse of Design:}
Consider now the pseudoinverse term $\|\MTM^+\|$ when $\bphi_{m+1}$ is adjoined to $\MTM(m)$ to create  $\MTM(m+1)$. Theorem~\ref{thm:add-column-inverse-norm} guarantees that whenever $\bphi_{m+1}$ is linearly independent of the old model (does not lie in the column space of $\MTM(m)$), then the induced norm of the new pseudoinverse is bounded below by the induced norm of the old pseudoinverse:  
\[
\|\MTM(m+1)^+\| \ge \|\MTM(m)^+\|.
\]
Similarly, when $\bphi_{m+1}$ is linearly dependent on the old model, then the induced norm of the new pseudoinverse is bounded above by the norm of the previous pseudoinverse
\[
\|\MTM(m+1)^+\| \le \|\MTM(m)^+\|.
\]
This proves Theorem~\ref{thm:main}.

\paragraph{Limiting behavior of \texorpdfstring{$\bbA$}{A}}\label{methods:limit-of-A}
As the number $m$ of model parameters gets large, the norm $\|\bbA\|$ is dominated by the norm of the pseudoinverse $\|\MTM^+\|$.   For purposes of this analysis, assume that the columns of the training-set (top) part of $\bbM$ are independent identically distributed (i.i.d.) random vectors $\bphi_i\in \R^n$ with finite second moment $\ev[\bphi_i \bphi_i\trp] = \Sigma$, where $\Sigma$ is of full rank (rank $t$).   

The Strong Law of Large Numbers guarantees that 
\begin{align*}
\frac{1}{m} &\MTM(m) \MTM(m)\trp \\
&= \frac{1}{m} \sum_{i=1}^m \bphi_i\bphi_i\trp \stackrel{a.s.}{\longrightarrow} \ev 
[\bphi_i\bphi_i\trp]\\
& = \Sigma
\end{align*}
as $m\to \infty$.  This implies that the smallest eigenvalue of 
$\frac{1}{m} \MTM(m)\MTM(m)\trp$ converges almost surely to the smallest eigenvalue $\lambda_{\mathrm{min}}>0$ of $\Sigma$.  Thus the smallest eigenvalue of $\MTM(m) \MTM(m)\trp$ approaches $m \lambda_{\mathrm{min}}$ and goes to infinity almost surely as $m\to \infty$.
Thus the smallest singular value $\sigma_{\mathrm{min}}(m) = \sqrt{\lambda_{\mathrm{min}}(m)}$ of $\MTM(m)$ also goes to infinity, and this implies
$
\|\MTM(m)^+\| = \frac{1}{\sigma_{\mathrm{min}}} \stackrel{a.s.}{\longrightarrow} 0$.

Because $\|\MTU(m)\|$ is bounded above and decreasing in $m$, we have 
\begin{align*}
\|\bbA(m) \| &=  \|\MTM(m)^+ \MTU(m)\|\\
& \le \|\MTM(m)^+\| \|\MTU(m)\| \stackrel{a.s.}{\longrightarrow} 0.
\end{align*}
In the special case that the columns $\bphi_i$ are i.i.d.~standard normal and $m>n$, it is known \cite[Thm 2.6]{Vershynin} that $\ev[\sigma_{\mathrm{min}}(m)] \ge \sqrt{m} - \sqrt{n}$, so $\|\MTM(m)^+\|$ and $\|\bbA(m)\|$ are $O(m^{-1/2})$ or smaller.

\subsubsection{Model--Data Trade-off}
\label{sec:overmodeling}

Because the data and model insufficiency terms are non-decreasing and non-increasing respectively, there is an inherent trade-off between them.
To study this trade-off, we introduce the combined model and data insufficiency error: 
\begin{align}
  \
  \| \bbE_I \btheta \|^2 & = \| {\color{Green}{\bbP_{\K(m)}}} \btheta_{\M(m)} \|^2 + \| {\color{red}{\bbI_{\U(m)}}} \btheta_{\U(m)} \|^2.
\end{align}
where we have made the $m$-dependence explicit.
To make statements about the dependence of the combined insufficiency errors, we consider two prior distributions for the distribution of the components of $\btheta$.

We first consider  the \emph{random feature model} in which components of $\btheta$ are i.i.d.~random variables with mean zero and variance $\sigma^2$. (Note that the random feature model is sensible only when the parameter space has finite dimension, otherwise the norm $\|\btheta\|$ would be infinite.)
In this setting the expected total insufficiency error is
\begin{align}
\ev_\theta\left[ \|  \bbE_I  \btheta \|^2 \right] & =  \sigma^2 \left( \Tr \bbPN + \Tr \bbIU \right) \\
& = \sigma^2 ( \dim \K  + \dim \U ).
\end{align}
At each step $\dim(\U)$ decreases by one, while $\dim(\K)$ increases by either zero or one; so, for the random feature model, the expected total insufficiency error is a strictly nonincreasing function of $m$.

In many ways the random feature model is unrealistic for scientific and engineering applications.  
Modelers often have prior information about which parameters are most important and preferentially order the parameter vector to reflect this.
In such cases and for very small $m$, as $m$ increases there is often an initial descent of model insufficiency due to the model's rapidly increasing ability to capture the signal faithfully.
This is conceptually analogous to reducing bias in the classical bias--variance paradigm. 
However, for very large models, the data insufficiency grows faster than the decrease in model insufficiency.
This growing error for large models is not analogous to variance and cannot be termed over-fitting.
Rather, it reflects the lack of invertibility for large models, specifically, larger parameter bias as more of the mass of $\btheta$ is projected into the kernel $\K$ of the design matrix $\MTM$.

This phenomenon of growing data insufficiency could be thought of as a form of  \emph{over-modeling}.
It occurs when parameters that are expected to contribute minimally to the signal are included in the model.
To be accurately inferred, such parameters place stringent informativity requirements for the data, amplifying the effects of data insufficiency.
This growing data insufficiency is less of a problem in random feature models, because all parameters contribute  essentially equally, which is why, as shown in the discussion above about random feature models and in the examples in Section~\ref{sec:realworld} below illustrate that random feature models tend to have optimal performance in the asymptotic limit as $m\to \infty$.  
But models that exploit prior information, so that the ordering of the basis functions and/or the choice of the training points are physically motivated, are more likely to suffer from increasing data insufficiency as the number of parameters grows.  Thus these models generally  have their ideal risk occur in the classical regime.

\section{Demonstrations and Applications}
\label{sec:applications}

\subsection{Random Feature Models}\label{sec:RandomExamples}

Although the motivating example in Section~\ref{sec:aliasing} for the generalized aliasing decomposition (GAD) was focused on one-dimensional polynomials, the GAD applies much more generally to the problem of fitting a function $f:\Omega  \to \R$ or $f:\Omega \to \C$ for a general set $\Omega$.  We illustrate this with examples of two different choices of models applied to three different data sets.  
The two bases are random Fourier features  (RFF) and random ReLU features (RRF) (described in \cite{BelkinOriginal} and \cite{MeiMontanari}).  

\begin{figure*}
    \begin{center}
    \includegraphics[width=\textwidth]{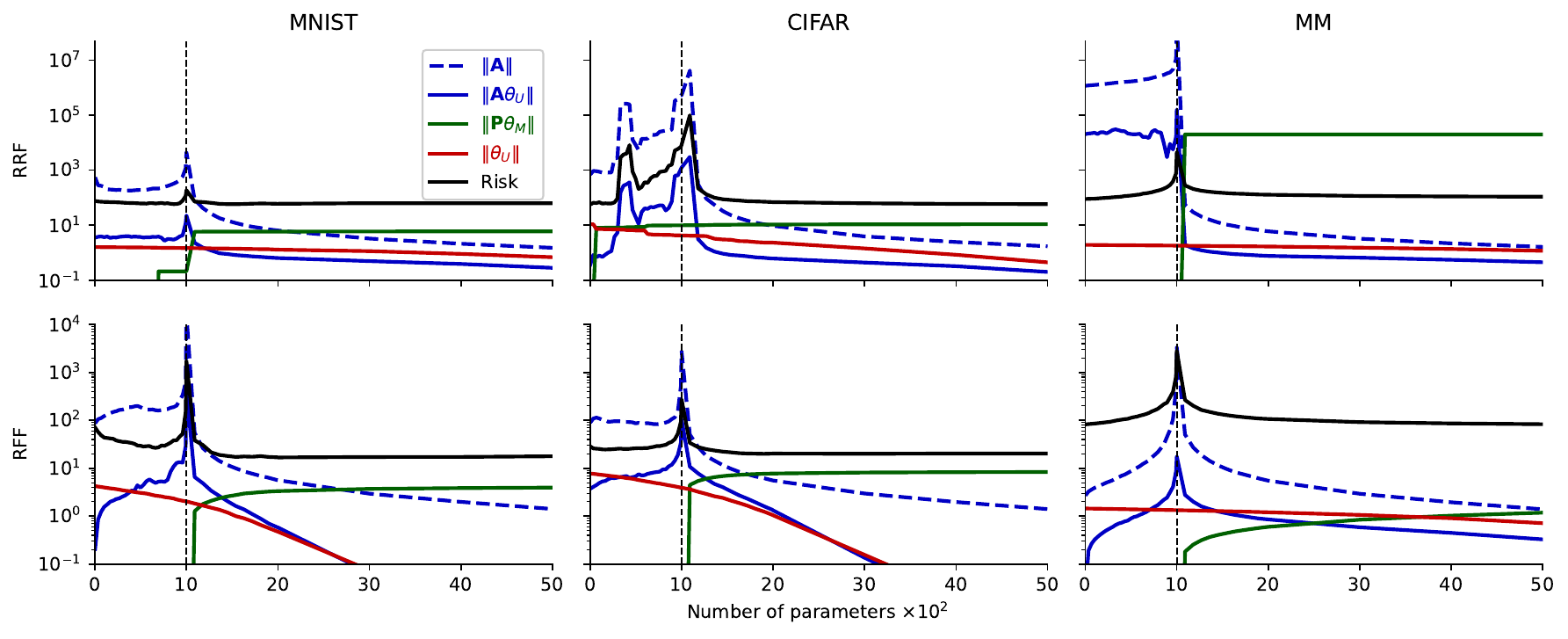}
    \end{center}
\caption{The induced (spectral) norm of the aliasing operator ${\bbA}$ (dashed blue), the aliased parameter error $\|\bbA \btheta_\U\|$ (solid blue), the data insufficiency parameter error $\|\bbPN \btheta_\M\|$, the model insufficiency parameter error $\|\bbIU\btheta_\U\|$, and the risk (black) for the random ReLU features (RRF) model (top row) and the 
random Fourier features (RFF) model (bottom row) on the MNIST and CIFAR-10 datasets, as in \cite{BelkinOriginal}, as well as on the Mei-Montanari (MM) sphere \cite{MeiMontanari}. In each case the models were trained on 1,000 randomly chosen training points (vertical dashed black line) with the number of modeled parameters ranging from 1 up to 5,000. Although the aliasing operator $\|{\bbA}\|$ and aliasing parameter error $\|\bbA \btheta_{\U}\|$ both go to zero almost surely as the number of parameters goes to $\infty$, the full parameter error has contributions from $\bbIU\btheta_\U$ (also decreasing, but slowly) and the data insufficiency parameter error $\|\bbPN\btheta_\M\|$, which, though bounded above by $\|\btheta\|$, is nondecreasing.
Data insufficiency $\|\bbPN\btheta_\U\|$ is generally $0$ until the interpolation threshold, but in the top center panel (and to a lesser degree in the top left panel) it is nonzero before the interpolation threshold, indicating that the design matrix $\MTM$ fails to be full rank fairly early.  Presumably this happens because ReLU is vanishes for many inputs.  The early large decrease in  $\|\bbA\|$ in that top center panel could be partly due to adding linearly dependent columns to $\MTM$ or to the (usually much less significant) decrease in $\|\MTU\|$ as columns are moved from $\MTU$ into $\MTM$.}
\label{fig:belkin-MM-plot}
\end{figure*}

All of these basis functions are of the form $\phi_k(\bt) = \sigma(\langle \bt, \bv_k \rangle)$, where the $\bv_k\in \R^d$ are i.i.d.~normal, 
and $\sigma$ is some activation function. 
In the case of the RRF model, the activation function is the usual ReLU, and in the case of RFF the activation function is $\sigma(x) = \exp(i \pi x)$.  The models that result from using these two choices (either RRF or RFF) can both be thought of as 2-layer neural networks of the form 
\[
y(\bt) = \sum_{k=1}^m \theta_k \phi_k(\bt) = \sum_{k=1}^m \theta_k \sigma(\langle \bt, \bv_k \rangle).
\]

The data sets we use here are images from MNIST and CIFAR-10 and points from the sphere $\mathbb{S}^{d-1}(\sqrt{d})$, as in Mei--Montanari \cite{MeiMontanari}; we have arbitrarily fixed $d=1024$ for this sphere.  In each case 1,000 training points $\bt_i$ were drawn uniformly and evaluated at 6,000 basis functions (either RRF or RFF).  
The columns of the resulting design matrix $\MTM$ and unmodeled block $\MTU$ are all of the form 
$\bphi_k = (\phi_k(\bt_1), \dots, \phi_k(\bt_n))$. 

In Figure~\ref{fig:belkin-MM-plot} we plot the norm of the aliasing matrix $\bbA$ and the parameter error contribution from each of: data insufficiency, model insufficiency, and generalized aliasing.  
The risk is also plotted  and each of these terms are displayed as functions of the number $m$ of parameters for these models on the three different datasets.
Recall that operator norms $\|\bbPN\|$ and $\|\bbIU\|$ are always $1$ or $0$, so we instead plot the norms of the products $\|\bbPN \btheta_{\M}\|$ and $\|\bbIU \btheta_{\U}\|$, which are the contributions to parameter error due to data insufficiency and model insufficiency, respectively.  We also plot $\|\bbA \btheta_{\M}\|$, to show the effect of each part of the GAD on the parameter error.  

The GAD decomposition in these examples closely matches the canonical picture presented in Figure~\ref{fig:AliasingDecomposition}, illustrating the dominant effect that the aliasing operator has on the non-monotonic behavior of the full risk.  This generic behavior is because the random selection of additional features (columns in $\MTM$) almost surely guarantees that such new features are linearly independent (on the sample points) from the existing features up to the interpolation threshold so that the risk will increase with model complexity.  After the interpolation threshold, additionally added features will be linearly dependent on the existing modeled features, and the aliasing (and hence risk) will decrease with model complexity.

\subsection{Why call it `aliasing'? Discrete Fourier series}\label{sec:Fourier}

To clarify the name ``generalized aliasing,'' we turn to an example familiar in the signals-processing community, the Fourier decomposition, which we describe here.

For a square-integrable function on the interval $[0,T]$ we will assume that our training data comes from equally spaced points $0 = t_0 < \cdots < t_n = T$.
We let  $\omega_n = \exp(2\pi \imath / n),$
be a primitive $n$-th root of unity and introduce the Fourier basis vectors $\w_n^{(k)} =(\omega_n^0,\omega_n^k,\ldots,\omega_n^{(n-1)k})$.  The discrete Fourier transform is the vector of coefficients $\hat{\f} = (\hat{f}_0,\hat{f}_1,\ldots, \hat{f}_{n-1})$ such that
\begin{equation}
    \f = \sum_{k=0}^{n-1} \hat{f}_k \w_n^{(k)},
\end{equation}
where the vector $\f$ is the vector of the sampled values of the function $f(t)$ sampled at the specified sample points.  Orthonormality of the Fourier basis in the standard $\ell^2$ inner-product space allows us to identify the Fourier coefficients
\begin{equation}
    \hat{f}_k = \left\langle \w_n^{(k)},\f\right\rangle = \frac{1}{n} \sum_{\ell=0}^{n-1}\omega_n^{-k\ell} f_\ell.
\end{equation}

\begin{figure*}
\begin{center}
    \includegraphics[width=0.49 \textwidth]{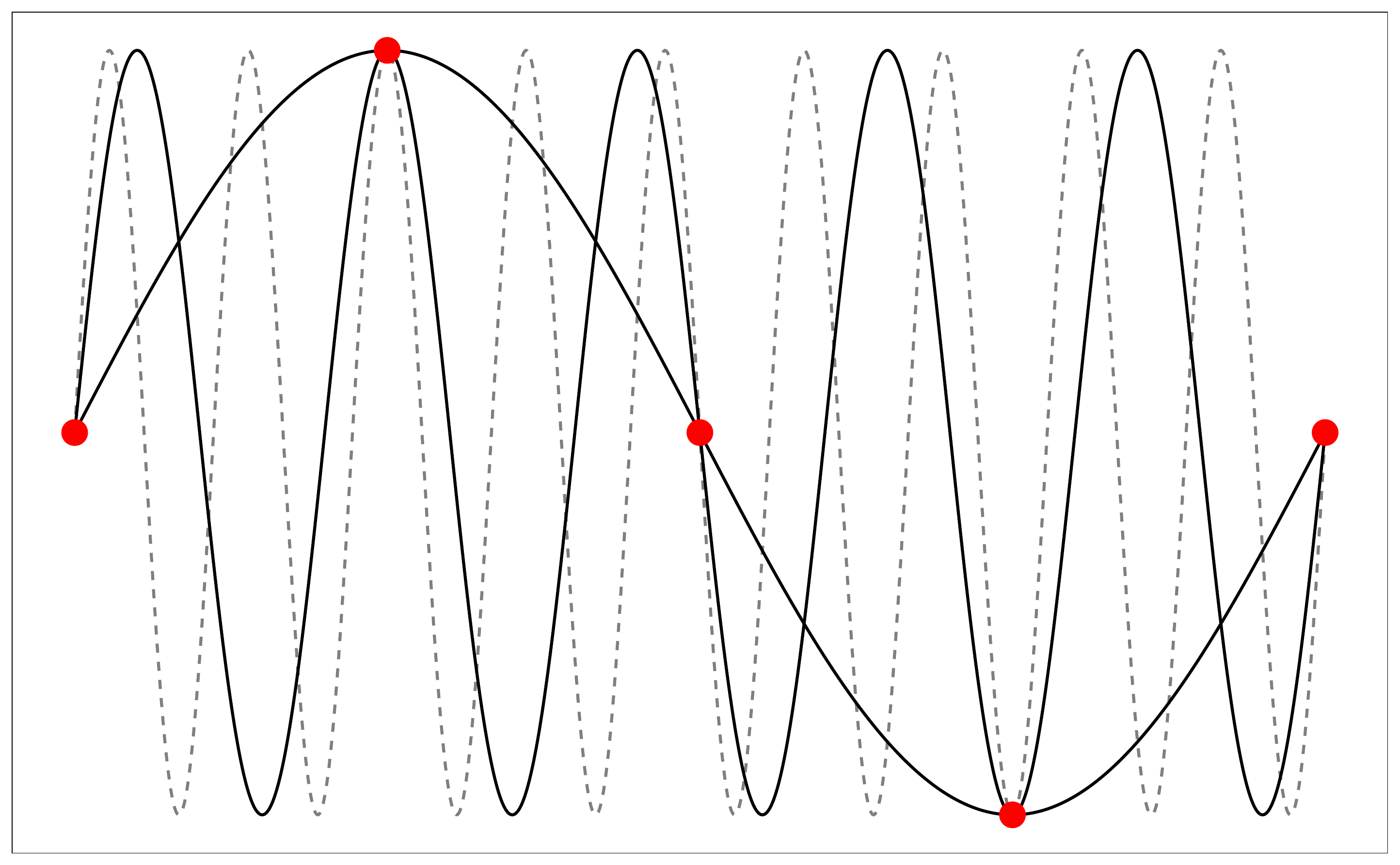}\includegraphics[width=0.49\textwidth]{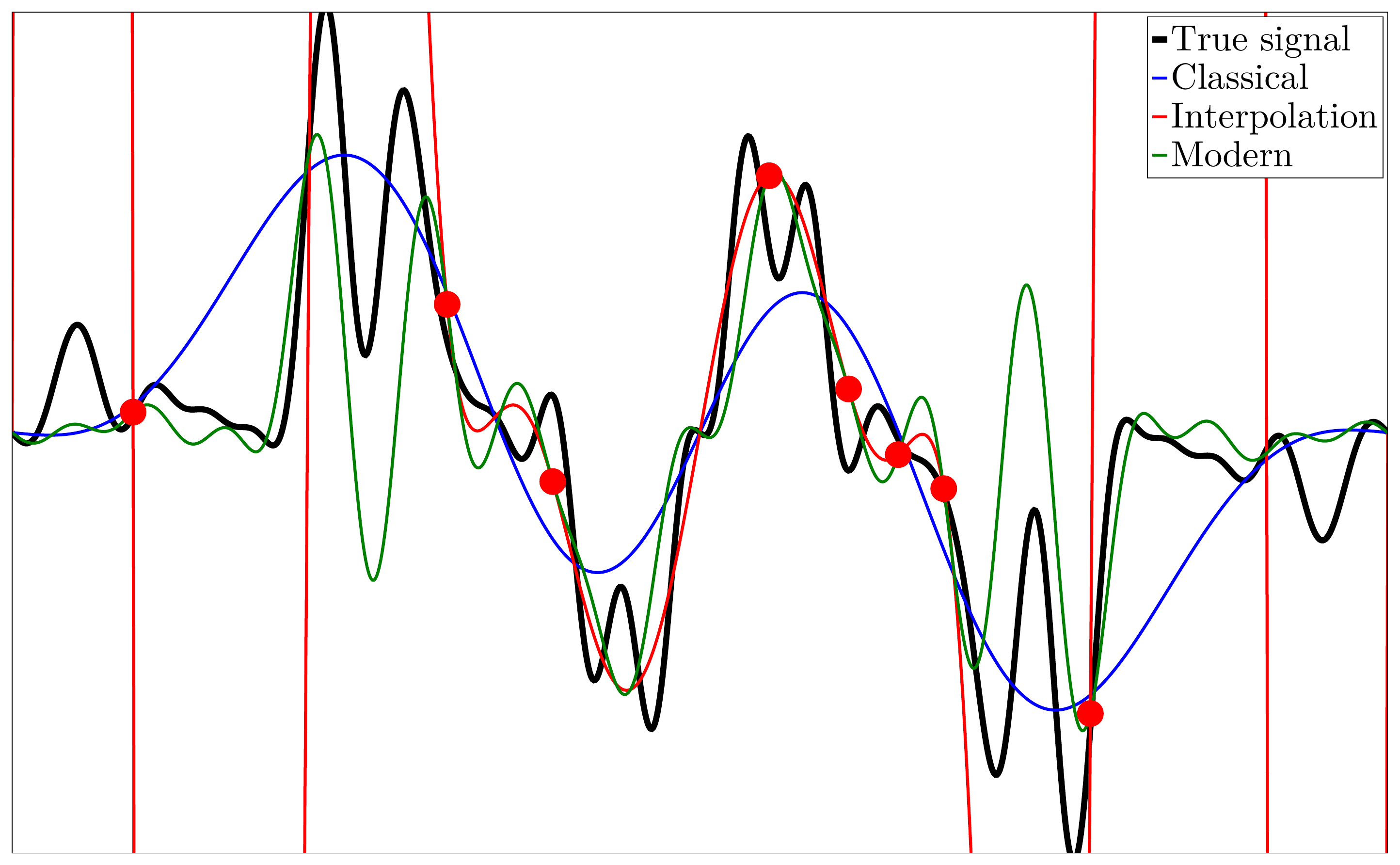}
\end{center}    
    \caption{Aliasing occurs when basis functions that are independent over the entire domain are linearly dependent at the sampled points (left).  
      When fitting a noisy signal (right), the classical sweet spot includes the dominant modes in the signal (blue).
      Over-fitting occurs when the combined contribution from the unmodeled modes is aliased into the model parameters, producing wild swings in the model predictions (red).
      Including additional terms allows the learning algorithm to distribute that signal over several basis terms.
      The result is a model whose predictions oscillate rapidly on a scale that is statistically similar to the true signal (green).}
    \label{fig:FourierAliasing}
\end{figure*}

In this example (to illustrate the signals-processing version of aliasing) we select the same number of basis functions $n$ as training points, that is  $m=n$.  The testing points are all other points in the interval. 
The design matrix is a variant of the Vandermonde matrix 
\begin{equation}
    \MTM = \frac{1}{n}\begin{pmatrix}
        1 & 1 & 1 & \cdots & 1\\
        1 & \omega_n^{-1} & \omega_n^{-2} & \cdots & \omega_n^{-(n-1)}\\
        1 & \omega_n^{-2} & \omega_n^{-4} & \cdots & \omega_n^{-2(n-1)}\\
        \vdots & \vdots & \vdots & \ddots & \vdots\\
        1 & \omega_n^{-(n-1)} & \omega_n^{-2(n-1)} & \cdots & \omega_n^{-(n-1)^2}
    \end{pmatrix},
\end{equation}
and the unmodeled block  $\MTU$ is bi-infinite with $n$ rows and columns. 
Since  $\omega_n^{\ell n} = 1$ for any integer $\ell$, the unmodeled block $\MTU$ is equal to an infinite number of copies of the design matrix
\[
\MTU = \begin{pmatrix}\cdots & \MTM & \MTM & \cdots \end{pmatrix}.
\]
Because we have selected $m=n$, the design matrix is full rank, and $\MTM^+ = \MTM^{-1}$.  
Thus $\bbA = \MTM^{-1} \MTU$ and $\bbB = \bbI_n$. 
This gives 
\begin{align}
    \bbA 
    &= \MTM^{-1} \begin{pmatrix} \cdots & \MTM & \MTM & \cdots \end{pmatrix} \\   
    \begin{pmatrix}
      \cdots &   \bbI_n & \bbI_n & \bbI_n & \cdots
    \end{pmatrix};
\end{align}
that is, $\bbA$ is a bi-infinite matrix  with $n$ rows and infinitely many columns in both directions, and it corresponds to infinitely many copies of the $n\times n$ identity matrix $\bbI_n$.

This derivation aligns exactly with the traditional concept of aliasing in the signals-processing literature \cite{roberts1987digital}, where the first column of the $\ell$-th copy of $\bbI_n$ in $\bbA$ corresponds to the $\ell n$th mode of the system, which is exactly aliased to the $0$-th mode; the second column of each copy of $I_n$ corresponds to the $(\ell n+1)$-th mode which is exactly aliased to the first mode of the actual signal, and so forth.  Unless the signal is band-limited, an infinite number of modes are aliased to each of the modeled modes.  Traditionally, the aliasing effect is not significant because signals are assumed to have most of their strength in the lower frequencies, that is the magnitude of the higher modes $\theta_k$ is assumed to decay to $0$ rapidly as $k\rightarrow \pm\infty$, which means that although $\bbA$ is bi-infinite, its effect is minimal on the actual representation of the signal.

This mathematical derivation is represented visually in the left panel of Fig.~\ref{fig:FourierAliasing}: although basis functions are independent over the entire prediction domain, they may make identical predictions over the sampled subset (red dots).
If the true signal contains contributions from all basis functions, but only a subset is explicitly modeled, the contribution from the unmodeled modes is aliased into the truncated representation.

The right panel shows three fits for an artificial data set using the Fourier basis.
The true signal (black) includes contributions from all Fourier modes (although the low-frequency modes dominate).
The classical sweet spot (blue) only models the dominant modes and produces a reasonable interpolation.
At the interpolation threshold (red), however, the aliasing operator magnifies the unmodeled modes, producing large swings in the model predictions between the training samples.
Beyond the interpolation threshold (green), the additional, high-frequency basis elements temper the aliasing effects by redistributing the signal among multiple basis functions.
The result is a rapidly oscillating signal that does not exhibit the wild swings of overfitting.
Although the oscillations in this inferred signal do not match those of the true signal, they are statistically similar, leading to reasonable model predictions.

\subsection{Differential Equations}
\label{sec:collocation}

Despite their superficial dissimilarity, it has been recognized for decades that solution methods for differential equations are formally equivalent to data fitting problems, as we see here.
A linear ordinary differential equation can be written as $\mathcal{L}[u](x) = f(x)$ where $\mathcal{L}$ is a linear differential operator, $u(x)$ is the unknown function, and $f(x)$ is a given function, often referred to as the ``data.''
As an example in this section we use the simple case where $\mathcal{L}[u] = u''(x)$, which describes the transverse displacement of a string under tension with transverse loading force given by $f(x)$.  The fundamental concepts, however,  are much more general than this simple example.

To solve such equations numerically, many approximation schemes exist in which $u(x)$ is approximated in some finite-dimensional subspace, such as with finite-differences, Galerkin truncation, or a collocation method.
In nearly all cases, the schemes lead to computational problems formally equivalent to the regression problem described in section~\ref{sec:aliasing}, which we now demonstrate explicitly for collocation.
Expand $u(x) = \sum_{\ell} \theta_{\ell}\psi_{\ell} (x)$ in some basis $\{\psi_{\ell}\}$; the differential equation then becomes $\sum_{\ell} \theta_{\ell} \mathcal{L}[\psi_{\ell}](x) = f(x)$.
In the collocation approach, we enforce that this equation is satisfied exactly at several sample points (training points) $x_i$, so that $\mathcal{L}[u](x_i) = f(x_i)$ for each training point.
If we denote $\phi_{\ell}(x) = \mathcal{L}[\psi_{\ell}](x)$, these conditions take the form:
\begin{align}
  \label{eq:collocation}
  \sum_{\ell} \theta_{\ell} \phi_{\ell}(x_i) & = f(x_i), \quad i = 1, \dots, n.
\end{align}
The collocation problem is then to choose basis functions $\psi_{\ell}(x)$ (and by extension $\phi_{\ell}(x)$) and collocation points $x_i$ such that solving Eq.~\eqref{eq:collocation} leads to as small error as possible throughout the entire domain.
This problem formulation is equivalent to a regression problem and the generalized aliasing decomposition gives insights into the structure of the errors.

To make these ideas concrete, consider the specific case of
\begin{align}
  \label{eq:ODEexample}
  u''(x) & = x \\
  u(0) & = u(\pi) = 0,\notag
\end{align}
which has the solution $u(x) = (x^3 - \pi^2x)/6$.

We first solve this problem using a sine basis $\psi_{\ell}(x) = \sin({\ell}x)$ with 32 uniformly spaced collocation points and 32 validation points uniformly spaced between them.
The resulting GAD for this problem is shown in the upper left of Figure~\ref{fig:ODEexample}.
Because the Fourier basis is orthonormal with respect to the uniformly spaced points, the aliasing operator has unit norm.
This scenario is equivalent to the traditional understanding of aliasing as presented in section~\ref{sec:Fourier}, and indeed, aliasing has functionally no effect on the risk curve.  The elimination of $\bbA$ as a factor in the risk arises because this selected basis is composed of exactly the orthonormal eigenfunction basis of the Sturm-Liouville problem defined by Eq.~\eqref{eq:ODEexample}.

For this specific, well-adapted basis the risk curve has a weak ``U'' shape that reflects the trade-off between the data and model insufficiency.
The minimum occurs precisely at the interpolation threshold where these two contributions to the error are balanced, demonstrating why the Fourier transform (and the sine series employed here) are most optimal at the interpolation threshold.
The striking absence of any aliasing effects is because the basis is optimally adapted to the sampling points.
However, this result is sensitive to many aspects of the problem formulation, as we now explore through the lens of the generalized aliasing decomposition.

\begin{figure*}
    \begin{center}
    \includegraphics[width=\textwidth]{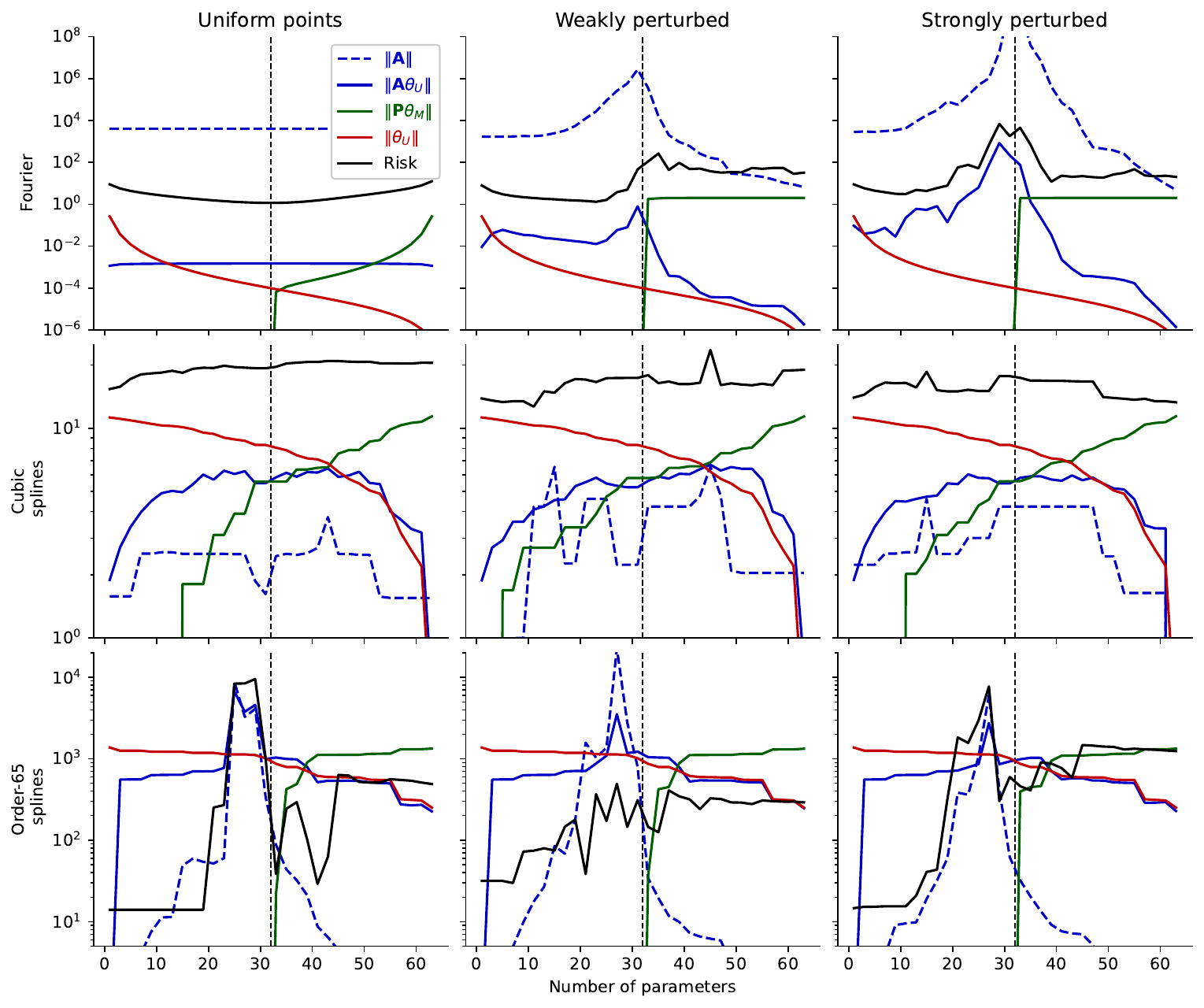}
    \end{center}
    \caption{Generalized aliasing decomposition for numerical solutions to the differential equation~\ref{eq:ODEexample} using different bases and sampling schemes. In each case there are 32  collocation (training) points and 32 validation points.  The dashed vertical black line marks the interpolation threshold.  In the first column, the sample points (collocation and validation)are uniformly spaced.     
    Moving to the right, the sample points are increasingly perturbed by a random amount. 
    The first row is a Fourier-sine basis.  Because these basis functions are orthonormal at the uniform points, the norm $\|\bbA\|$ of the aliasing operator is essentially constant in the upper left panel, but as the sample points are perturbed, the basis functions move away from being orthonormal and aliasing increases near the interpolation threshold. 
    The second row corresponds to a basis of cubic splines. These have fairly narrow, compact support, so the basis functions are close to being orthogonal on all the sample points, keeping the aliasing small across the row, but the small support means model insufficiency $\|\bbIU\btheta_\U\|$ drops off much more slowly than in the Fourier case, and it also causes the data insufficiency $\|\bbPN \btheta\|$ to become significant long before the interpolation threshold.  
    The bottom row corresponds to higher-order splines.  These have support across the full domain, which makes for nontrivial aliasing for all the different choices of sample points, but the data insufficiency is very small until the interpolation threshold. These high-order splines are also highly localized, which means the model insufficiency $\|\bbIU \btheta_\U\|$ drops off more slowly than in the other rows. See Section~\ref{sec:collocation} for more details about this example.}
    \label{fig:ODEexample}
\end{figure*}

Moving horizontally across the top row in Figure~\ref{fig:ODEexample} we explore the sensitivity of the solution to the choice of collocation and validation points.
If the sampling points (training and validation) are weakly perturbed from uniform spacing (top row, middle), aliasing  emerges around the interpolation threshold along with the characteristic double descent peak. 
Because the exact solution is continuous, the Fourier series converges rapidly; corresponding to a rapidly decreasing model insufficiency $\|\bbIU\btheta\|$.
But the aliasing contribution $\Vert \bbA \theta_\U\Vert$ to the parameter error (and hence to risk) is large.
Consequently, the optimal solution is no longer at the interpolation threshold; it occurs at the classical sweet spot.
In this case, the minimum is to the left of the interpolation threshold rather than to the right, because the basis is ordered with dominant terms first.
Consistent with our analysis in section~\ref{sec:overmodeling}, the asymptotic limit is suboptimal due to data insufficiency.

The second row of Figure~\ref{fig:ODEexample} shows the results for a cubic b-spline basis with 65 uniformly spaced knots throughout the domain.
We omit the two basis functions that do not satisfy the boundary conditions for a basis of 64 functions.  
Because there is no natural ordering, we randomly shuffle the basis functions, making the model equivalent to a random feature model.
Moving horizontally across the second row of the figure shows the GAD for the same choice of sampling points as for the Fourier basis.

In all three cases, notice that the contributions from aliasing are absent.
This is because the basis functions have relatively small compact support, which strongly limits their ability to alias with each other.
We observe similar results for other choices of bases with relatively small compact support, such as Haar wavelets.
But, unlike the other rows, in this basis data insufficiency $\|\bbPN \btheta_\M\|$ becomes nonzero long before the interpolation threshold.
This is also at least partly due to the relatively small compact support---the basis functions are each zero throughout most of the domain, so they have a nontrivial nullspace for our choice of sampling points.
This  phenomenon is similar to that observed for the random ReLU basis on CIFAR-10 in the upper middle panel of Figure~\ref{fig:belkin-MM-plot}.

Finally, on the bottom row we apply another b-spline basis of 65th-order polynomials.
As before, we remove the two basis functions that do not satisfy the boundary conditions and shuffle the remaining basis functions.
While these splines are technically continuous and non-zero throughout the entire domain, each basis function is strongly peaked around a small portion of the domain.
In this case, we see some contributions from aliasing before the interpolation threshold, but it is not as prominent as with the Fourier basis and is essentially independent of the choice of sampling points.  

In general, a similar analysis can be applied to other differential operators, including partial differential equations.
Indeed, the effects of aliasing (and the need for dealiasing) are well known in the simulation of nonlinear partial differential equations (see \cite{orszag1971elimination} for the original reference or \cite{boyd2001chebyshev} for a more thorough discussion).
The current decomposition applies there as well, and some results are also known for nonlinear equations, as we now summarize.

The generic setup for a quadratic nonlinearity would be of the form
\begin{equation*}
    \frac{d\y}{dt} \approx \y\odot \y,
\end{equation*}
where $\odot$ is the Hadamard (entrywise) product.
In this setting, the labels $\y$ denote a spatially and temporally dependent function described by the basis functions in the design matrix $\bbM$, i.e. $\y = \bbM\btheta$.  To solve this system, we note that the differential equation can be written as:
\begin{equation*}
    \bbM \dot{\btheta} \approx(\bbM\btheta) \odot (\bbM\btheta),
\end{equation*}
where the $\dot{}$ refers to the time derivative.  If the entire basis $\bbM$ could be used, then the solution is obtained by multiplying on the left by the appropriate pseudoinverse $\bbM^+$.  In reality, all of $\bbM$ is not available, and so we decompose the system just as before, leading to a term on the right-hand side that resembles the aliasing operator, but now with a quadratic dependence on the unmodeled terms $\MTU$ which leads to a famous ``$3/2$ rule'' for pseudospectral methods \cite{orszag1971elimination}.

\subsection{Material Discovery: Cluster Expansion}
\label{sec:realworld}

\begin{figure*}
\begin{center}
\includegraphics[width=\textwidth]{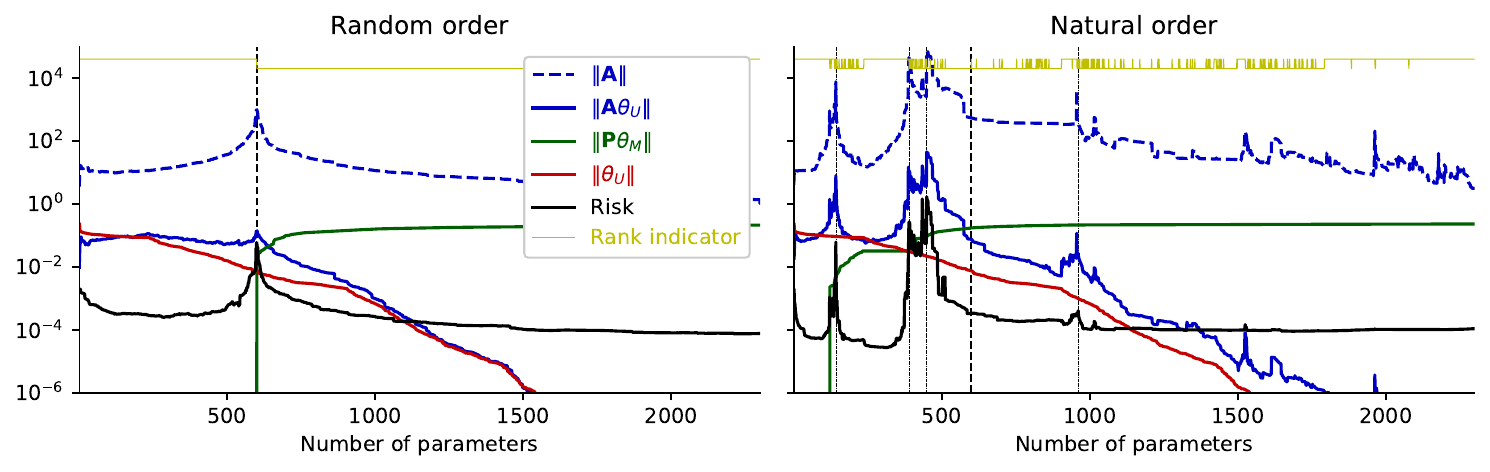}
\end{center}
    \caption{Norm of the operators and true risk (black) of the cluster expansion model of Section~\ref{sec:realworld}  as model complexity is increased (i.e., as parameters are added) with 600 training points. The interpolation threshold is indicated by a vertical dashed line. The yellow at the top of each panel is an indicator function that is high when the added basis function is linearly independent of the previous columns (restricted to training points) and low when it is linearly dependent.
    In the left panel rows and columns of $\mathbf{M}$ have been randomly ordered. In the right panel the rows and columns of the design matrix are given a ``natural'' ordering, resulting in multiple peaks (indicated by dash-dot vertical lines) and valleys for the risk and aliasing.  Note also how the data insufficiency $\|\bbPN\btheta_\M\|$ becomes significant only at the interpolation threshold for the random ordering, but it becomes significant long before the interpolation threshold in the natural ordering.}
    \label{fig:CEexample}
\end{figure*}

As a final example we consider the \emph{cluster expansion}, an extremely efficient model for prediction of novel materials phases. For a gentle but thorough introduction to the formalism, see~\cite{lerch2009uncle}. In brief, the cluster expansion model is a generalized Ising model \cite{sanchez1984generalized,sanchez1993cluster,sanchez2010cluster,zunger1994statics,van2002alloy,mueller2009bayesian,aangqvist2019icet,seko2006prediction} that in typical applications has hundreds to thousands of data points and a dozen to hundreds of inferred parameters. The prototypical application of the cluster expansion is predicting the formation enthalpy of an alloy as a function of elemental composition and configuration $\vec\sigma$. Eq.~\eqref{eq:CE} is a sum over different bonds (pairwise, three-way, etc) for every site in the crystal, an intuitive expression of physical chemistry.
\begin{multline}
\label{eq:CE}
E(\vec\sigma)  = J_0+\sum_{i} J_1 \xi_i + \sum_{\beta}^{\textrm{pairs}} \sum_{i,j} J_\beta \xi_i\xi_j +\\\sum_{\gamma}^{\textrm{triplets}}\sum_{i,j,k} J_\gamma \xi_i\xi_j\xi_k
+\sum_\nu^{\textrm{quads}}\sum_{i,j,k,l} J_4 \xi_i\xi_j\xi_k\xi_l + \cdots \\
= \sum_{\alpha}^\mathrm{clusters}J_{\alpha} \Phi_{\alpha}(\vec \sigma)\qquad
\end{multline}
where the ``bond'' indices run over all the possible sites, pairs of sites, triples, and so on.
The $J$'s are the ``bond strengths'' (inferred parameters, analogous to the $\theta$s in the notation above). $\vec\sigma$ is a vector of integers, the $i$-th component representing the type of atom sitting on the $i$-th lattice site.  The products of $\xi(\vec\sigma)$ functions\footnote{The \emph{site functions} $\xi$ themselves are usually discrete Cheybschev polynomials or a Fourier basis. Any functions that form an orthonormal set over the discrete values of $\sigma_i$ are suitable.} form an orthogonal basis
$\{\Phi_\alpha\}_\alpha$ in the discrete vector space of all possible atomic configurations. 

This model has a physically intuitive interpretation as representing chemical bonds between groups of atoms. For example, a product of two functions, $\xi_i\xi_j$,  represents a pairwise interaction between atoms on sites $i$ and $j$. The strength of the interaction is the magnitude 
$|J_{ij}|$, and the sign of $J_{ij}$ determines whether like or unlike atoms prefer to be ${ij}$-neighbors. 

The CE interactions $J_\alpha$ are typically inferred from quantum-mechanical energies. There are no obvious strategies for picking which alloy configurations to include in the training set. As to this question---the horizontal partitioning of $\bbM$ (deciding the which rows of $\bbM$ should be in $\MTM$ and which should be in $\MVM$ in Eq.~\eqref{eq:partition})---one often included the ``usual suspects,'' configurations that often occur in actual alloys; but this rarely provides enough information to generate a model with small generalization error.  

Choosing which basis functions to include in the model is even more difficult; it is difficult to know which physical interactions are negligible. (The vertical partition of $\bbM$ in Eq.~\eqref{eq:partition} divides the important interactions from those that are assumed to be negligible.) Many different strategies to address these two challenges have been employed in the CE community \cite{mueller2009bayesian,lerch2009uncle,aangqvist2019icet,hart2005evolutionary,PhysRevB.72.165113,seko2009cluster,nelson2013cluster,nelson2013compressive,PhysRevB.99.134206,avdw:atat,puchala2013thermodynamics,leong2019robust}.

Using the generalized aliasing decomposition, CE practitioners can now reason more effectively about how to make these two difficult choices---which configurations to sample and which basis functions to include in the modeling. Furthermore, the GAD elegantly explains the complex risk curves of a typical alloy system (see, e.g., Figure~\ref{fig:CEexample}). 
One can enumerate all possible configurations (up to some maximum number of atoms) \cite{hart2008algorithm,hart2009generating,hart2012generating}, identifying all the rows of the universe matrix $\bbM$ for this problem.  (In principle the number of rows is countably infinite, but under mild assumptions, one can enumerate all configurations up to a size that effectively includes all configurations that are likely to appear in nature.) It is also feasible to determine a complete set of basis functions \cite{sanchez1993cluster} for the enumerated configurations. \footnote{Until recently, enumerating all linearly independent basis functions, without also generating many linearly dependent basis functions, was an outstanding problem. This new algorithm is not yet published.} 

We demonstrate in the following study of a Pt-Cu alloy.
Choosing a realistic model size, we explain the resulting generalization curve through the lens of the GAD.
A binary alloy model containing up to ten unique atomic sites has 2346 unique configurations.\footnote{This formation enthalpy data were generated by ``unrelaxed'' Density Functional Theory calculations for configurations of platinum and copper, which were then adjusted by a linear regression using slight regularization to smooth out noise.} Figure \ref{fig:CEexample} shows the norms of the aliasing matrix, the model insufficiency, and the data insufficiency as a function of increasing basis size for a fixed number of training points.

Unlike the Fourier example, where a natural ordering is obvious, in this setting the ``right'' ordering is not clear. 
At first, we impose no assumptions about natural ordering to either the data or the basis functions parameters, randomizing rows and columns of $\bbM$. This is similar to the random feature models in Section~\ref{sec:RandomExamples}. 
The risk behavior (left panel of Fig.~\ref{fig:CEexample}) exhibits the prototypical double descent. Before the interpolation threshold, the behavior of the risk curve is the expected U-shape of the classical bias--variance trade-off. And beyond the interpolation threshold, the risk drops again.
As predicted in the random feature discussion in Section~\ref{sec:overmodeling}, the lowest-risk model is not at the classical sweet spot but in the asymptotic limit as $m\to \infty$. 

But cluster expansion practitioners do have some intuition about the natural order for the sample points and basis functions.
In their preferred ordering for cluster expansion, pair-wise interactions precede triplet interactions, and all triplet interactions come before any quadruplets, and so forth. Furthermore, the terms are ordered in each
class by diameter---short pairs before long pairs, small-diameter triplets before extended triplets, etc. 
This ordering is motivated by physical arguments that the strongest interactions are short-range and low body-order. For the ordering of the sample points (atomic configurations, defining rows of $\bbM$), there is the coarse guideline of ordering by ``size,'' denoted by the number of atoms in each configuration, but within each size class a natural ordering is not obvious.

This natural ordering of $n$-body/short-long was used to arrange columns and rows in $\bbM$ for the risk curve shown in the right panel of Fig.~\ref{fig:CEexample}. The aliasing $\|\bbA\|$ and the risk move essentially in unison and show a complicated behavior, neither the typical U-shape of classical bias--variance trade-off nor the basic double descent. Rather, the generalization curve has multiple peaks and valleys, whose positions correspond to locations where added basis functions transition from linear independence to linear dependence (marked by the yellow ``indicator function'' at the top of the plot).  Vertical dash-dot lines are included to clarify the connection between peaks in the operator norm $\|\bbA\|$ (dashed blue) and peaks in the risk (solid black).
 Surprisingly (for the bias--variance paradigm), the interpolation threshold does not seem to play any role in the generalization curve for this naturally ordered case. 

The peaks in the operator norm $\|\bbA\|$ for the naturally ordered case (right pane of Fig.~\ref{fig:CEexample}) suggest a simple improvement to ordering the columns of $\bbM$. The aliasing norm suggests that the lowest generalization error will happen around the 50-th parameter or near the 300-th parameter. Between these two, there is a group of parameters that drastically increase the aliasing (and  likely the error as well). By re-ordering the first 500 parameters, swapping the high-risk group with the group in the second ``valley'', the empirical risk will have a deep, broad valley for the first few hundred parameters. This gives practitioners a generous range of model sizes that avoid unexpected spikes in the empirical risk. 

Finally, note also that for the naturally ordered case, the optimal risk occurs at a classical ``sweet-spot,'' with a low number of parameters, and is better (lower) than the optimal risk in the randomly ordered case, which occurs in the asymptotic limit as the number of parameters grows large.

\section{Discussion}\label{sec:discussion}

We have demonstrated the utility of the GAD for explaining complicated, non-monotonic risk curves in a variety of different settings.  Now we turn our attention to using the GAD for improved model development and the pursuit of more efficient and accurate representations of the data.  We will discuss the impact this decomposition has on modeling and sampling decisions, and the influence those decisions have on the shape of generalized risk curves.

\subsection{General Insights into Modeling}

The formal analysis provided above, as well as the examples demonstrating the GAD give practical, intuitive guidance for formulating models which we discuss here.

\subsubsection{Choosing the Basis}
If the $n$ training points are known and fixed, a modeler can control the norm of $\MTM^+$ and $\bbA$ (and hence generically control the magnitude of the risk) by strategically choosing the basis functions, without knowing anything about the labels $\by$.

For example, consider what happens when we choose the first $n$ basis functions $\phi_k$ so that, when evaluated at the points $\bt_1, \dots, \bt_n$, the resulting vectors $\bphi_k = (\phi_k(\bt_1), \dots, \phi_k(\bt_n))$ are orthonormal. 
If the columns of $\MTM$ are the first $m \le n$ of these vectors, then the inverse $\MTM^+$ always has induced norm $\|\MTM^+\| = 1$.  In this situation the norm is constant as $m$ increases up to $n$; and then for $m>n$, no matter which additional columns are added, the norm $\|\MTM^+\|$ cannot increase and will eventually shrink to $0$ (almost surely).  
Thus, the product $\|\MTM^+\| \|\MTU\|$ in the upper bound 
\[
\|\bbA\| \le \|\MTM^+\| \|\MTU\|
\]
on the norm of $\bbA$ also can never increase with $m$.  Given a prior on the coefficients $\btheta$, if the basis functions are ordered to reflect the expected magnitudes of the corresponding coefficients, then we expect there to be no peak in $\|\bbA\btheta_\U\|$ at all---only descent.

In the discrete Fourier series example of Section~\ref{sec:Fourier}, the norm of $\bbA$ is always $1$ and does not decrease to $0$ because the columns of $\M$ are specially tuned to the training set to make $\MTU$ consist of infinitely many copies of $\MTM$.
This aligning of the basis functions to sample points explains why extreme over-fitting is rarely a problem in discrete Fourier transforms, in spite of it being formally equivalent to ordinary least squares regression at the interpolation threshold.

\subsubsection{Choosing Training Points}
If the basis functions are given and fixed, but the modeler has control over the choice of the training points, then they can control the norm $\|\bbA\|$ by strategically choosing the points $\bt_1, \dots, \bt_n$.  Again, this requires no knowledge of the labels $\by$.

For example, consider the case of fitting polynomial functions on the interval $[-1,1]$ with the Legendre basis, consisting of polynomials  $\{P_k\}_{k\in \N}$ which are orthogonal with respect to the inner product $\langle f, g \rangle = \int_{-1}^1 f(t) g(t) \, dt$, with  $P_k$ of degree $k$ and $P_k(1)=1$ for all $k$.  For a given number $m$ of model parameters (the first $m$ Legendre polynomials), if we are able to choose $n$ points at which to evaluate the basis functions, then choosing the points to be the $n$ Legendre--Gauss points, which are the zeros of $P_n$, gives much better results than choosing the points randomly (drawn uniformly).  This is shown in Figure~\ref{fig:Legendre}, where the randomly chosen training points make $\|\bbA\|$ many orders of magnitude larger than with the specially chosen Legendre--Gauss points.  
In this case a judicious choice of training points makes a huge difference to $\|\bbA\|$.  Except for very special choices of $\btheta$, this means the risk Eq.~\eqref{eq:GAD-risk} will also be substantially larger when the model is trained on random points than when it is trained on the specially chosen Legendre--Gauss points.

\begin{figure}
    \begin{center}
        \includegraphics[width=0.475\textwidth]{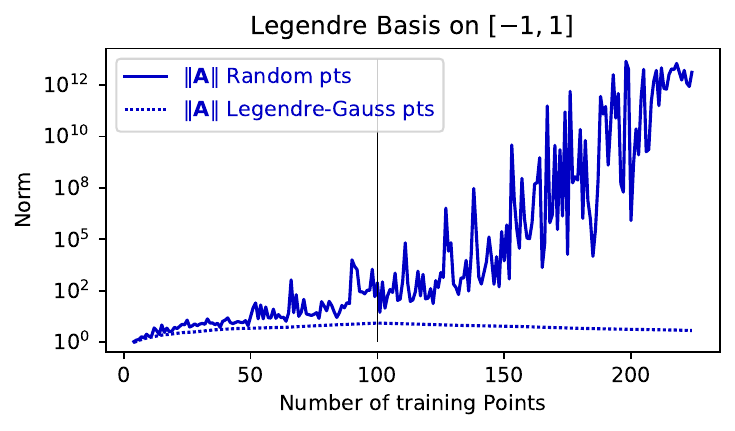}
    \end{center}
    \caption{The induced norm $\|{\bbA}\|$ of the aliasing operator for the Legendre polynomial basis with the model consisting of the first $m=100$ Legendre polynomials. The norms are plotted as functions of the number $n$ of training points, and the vertical black line indicates the interpolation threshold (note that this is inverted from the plots in the previous figures where $n$ is fixed and $m$ is variable).  The solid blue shows the norm of the aliasing operator when the training points are chosen randomly (drawn uniformly from $[-1,1]$), while the dotted blue shows the norm of the aliasing operator when the $n$ training points are chosen to be the Legendre--Gauss points (the zeros of the $n$th Legendre polynomial). The norm $\|\bbA\|$  for randomly chosen training points rapidly grows to be many orders of magnitude larger than for the Legendre--Gauss points.}
    \label{fig:Legendre}
\end{figure}

\subsubsection{Conditioning of \texorpdfstring{$\bbM$}{M}}
\label{sec:conditioning}

If $\bbM$ is poorly conditioned, then it is possible to have a relatively small error $\bbE_{\btheta}\btheta$ in the parameters that corresponds to a large error in the signal (large risk).  Thus it is desirable to select a basis that makes the full transformation $\bbM$ well conditioned. 

For polynomial approximation with the standard monomial basis $\{1,t,t^2,\dots\}$, the transformation $\bbM$ is a generalized Vandermonde matrix, which is very badly conditioned and generally should not be used with real-valued inputs \footnote{The Vandermonde matrix \emph{is} well conditioned in the special case that the inputs $t$ all lie on the unit circle $U^1 = \{t \in \C :  |z| = 1\}$, but it is badly conditioned if the inputs $t$ do not have unit modulus.}.  
But polynomial approximation for real inputs in the interval $[-1,1]$ is well conditioned with the Chebyshev polynomial basis or the Legendre polynomial basis.

\subsection{Regularization}

It has been observed that $L^2$-regularization (ridge regression) generally reduces the size of the peak in risk at the interpolation threshold, but it can also increase risk elsewhere along the curve \cite{MeiMontanari, yilmaz, nakkiran}.  This can be understood in terms of the impact of regularization on the GAD and the pseudoinverse of the design matrix.

For a given decomposition $\Theta = \M \oplus \U$ of the space $\Theta$ with $m =\dim \M$ model parameters, ridge regression  amounts to changing the objective from minimizing risk to minimizing
\begin{align}
\frac{1}{n}\| \by - \MTM \btheta_{\M}\|_2^2 + \lambda \| \btheta_{\M} \|_2^2, 
\end{align}
where $n$ is the number of training points and $\lambda$ is a user-chosen parameter.
It is straightforward to verify that the objective to minimize with $L_2$-regularization can be written as 
$\frac{1}{n}\left\| \widetilde{\by} - \tMTM \btheta_{\M} \right\|_2^2,\label{eq:2-regularize-as-LM}$
where $\tMTM = \left[\begin{smallmatrix}\MTM \\ \sqrt{n\lambda} \bbI_m\end{smallmatrix}\right]$ and $\widetilde{\by} = \left[\begin{smallmatrix}\by\\ \0 \end{smallmatrix}\right]$.  This changes the GAD to  to 
\[
\widetilde{\bbE}_{\btheta} = \begin{bmatrix}
I_\M  - \tMTM^+ \MTM & -\tMTM^+ \MTM\\
0 & I_\U
\end{bmatrix},
\]
so aliasing $\bbA$ becomes $\bbAt = \tMTM^+ \MTU$, and data insufficiency $\bbPN$ becomes $\bbPNt = I_\M  - \tMTM^+ \MTM$, while $\bbIU$ remains unchanged.
Expanding the product 
\[
\tMTM \tMTM\trp = \MTM\MTM\trp + n\lambda \bbI_m,
\] 
shows that every eigenvalue of $\MTM\MTM\trp$ is now increased by $n\lambda$ in this product.  Therefore, the singular values of $\MTM$ are all increased by $\sqrt{n\lambda}$ in $\tMTM$, and 
\begin{align*}
\|\tMTM^+\| 
&= \frac{1}{\frac{1}{\|{\MTM}^+\|} + \sqrt{n\lambda}}\\
&= \frac{\|{\MTM}^+\|}{1 + \sqrt{n\lambda} \|{\MTM}^+\|} \le \frac{1}{\sqrt{n \lambda}}.
\end{align*}

This bound is independent of both $m$ and $\MTM$, and 
it essentially removes the impact of any small singular values of $\MTM$ on the norms of $\tMTM^+$ and $\bbAt$. This explains why there is no significant peak in the risk at the interpolation threshold (or anywhere else, for that matter) for $L_2$-regularized (ridge regression) problems, provided $\lambda$ is sufficiently large. 
If $\sqrt{n \lambda} > \|\MTU\|$, then the norm of $\bbAt$ is smaller than the norms of data insufficiency and model insufficiency, which then dominate the parameter error.

The contribution $\|\bbPNt \btheta_\M|$ of data insufficiency to parameter error, however, can increase with regularization because it is no longer the projection of $\btheta_\M$ onto the null space $\K$ of $\MTM$ or $\tMTM$ but instead is 
\[
\|\bbPNt \btheta_\M\| = \|(\bbI_\M - \tMTM^+ \MTM) \btheta_\M\|.
\]
When $\lambda$ is large, the fact that  $\|\tMTM^+ \MTM\|\le \frac{\|\MTM\|}{\sqrt{n \lambda}}$, means that the data insufficiency operator approaches $\bbI_M$ and the contribution to parameter error from data insufficiency approaches $\|\btheta_\M\|$, which is generally larger than the projection $\|\bbPN \btheta_\M\|$ onto the kernel $\K$ of $\MTM$. 

Nevertheless, the norm of $\bbPNt$, while no longer necessarily bounded by $1$, is still bounded by 
\[
\|\bbPNt\| \le 1 + \frac{\|\MTM\|}{\sqrt{n\lambda}}.
\]

\subsection{How to think about the unmodeled signal  \texorpdfstring{$\MTU$?}{MTU}}
\label{sec:nescience}

The fundamental ansatz of generalized aliasing is the decomposition in Eq.~\eqref{eq:partition} that decomposes the signal into both modeled and unmodeled components.
Our conception of unmodeled signal is similar to ``noise'' as understood in classical and modern statistics.
Indeed, in comparing Eqs.~\eqref{eq:regression} and \eqref{eq:regression_nescience}, the unmodeled modes naively correspond to the noise in classical regression.
This decomposition may initially seem unnatural since any unmodeled components are unknown and consequently, difficult to reason about.
For some readers, this decomposition may seem an unnecessary introduction at best or an untractable complication at worst.
However, the concept is useful for distinguishing nuances in the modeling processes that are obscured by the traditional conception of statistical noise.

First, the heart of the GAD is recognizing that model representations are embedded within a universal function space.
This way of thinking is strongly motivated by signal processing, in which a signal is often assumed to have contributions from all modes, even if only a subset can be extracted from a finite sampling.
This enables us to quantify the trade-off between the modeled and the unmodeled contributions and their relative informativity.
In this way, the GAD naturally quantifies the intuition that as the model capacity grows, unmodeled signal necessarily shrinks.
In contrast, in the classical formulation, noise is modeled as a random variable whose scale parameter is not necessarily tied to the complexity of the model except as a tunable hyperparameter.
Thus, by quantifying the trade-off between the modeled and the unmodeled, we quantify the informational relationship between the data and the model.

Furthermore, recognizing the unmodeled allows flexibility in solving problems where random variables are not the natural representation.
For example, approximating the solution to differential equations is formally equivalent to regression.
However, considerable information is known about the analytic nature of these solutions, and it is often more natural to represent the unmodeled piece as another continuous signal from a set for which there is no natural measure.
In many applications, such as robust control, one is interested in worst-case scenarios.
In such cases, one takes the extremal case over the allowed set of unmodeled signals rather than an expectation value over a random variable.

Finally, our conception of unmodeled signal encompasses any limitations in representing either models or data.
Something as insipid as finite-precision arithmetic is a form of unmodeled signal that is not commonly equated with statistical ``noise.''
For example, consider the representation of a band-limited signal.
The Fourier sampling theorem guarantees its finite Fourier representation can be reconstructed from finite samples.
And when the signal is sampled at generic, random points
with an infinite precision representation, such a signal can still be exactly reconstructed.
In finite precision, however, the represented signal is no longer band-limited: round-off error introduces small, high frequency contributions.
Even if the high-frequency components introduced by the finite precision are bounded, aliasing greatly magnifies their impact on the reconstructed signal, and  the ill-conditioning of the aliasing operator leads to large errors in the inferred signal.

This final point reflects a much deeper philosophical issue when modeling data, which we summarize as \emph{fidelity} and \emph{sensitivity to representation.}
The concept of fidelity can be understood as the extent to which a representation is faithful to the real physical process.
Again, consider the example of Fourier analysis.
The utility of Fourier representations are often attributed to the fact that smooth functions have rapidly decaying Fourier series.
Consequently, a truncated Fourier representation of a smooth signal is faithful to the truth, in the sense that they are nearby in Fourier space.
Although the truncated series is formally wrong, its representational error is bounded.

In contrast, it is often possible to apply inaccurate approximations to models that nevertheless make accurate predictions.
When this occurs, a model exhibits insensitivity to the representation.
Arguably, the most famous and important example of this is the concept of \emph{irrelevance} in renormalization in statistical physics.
In renormalization, approximations are made not because they are accurate but because they do not affect observables of interest.
In Kadanoff's block-spin renormalization of the Ising model, groups of spins are aggregated into a single block-spin; that is, they are approximated as being perfectly correlated.
While such approximations are inaccurate for modeling spin correlations, they lead to good approximations of phase transitions.
The details of microscopic correlations are said to be \emph{irrelevant} to macroscopic observables.
On the other hand, the phase diagram is very sensitive to the so-called \emph{relevant} parameters, such as the applied field or temperature.
Small variations in these variables significantly impact the macroscopic order parameter.

The concepts of relevance in renormalization and sensitivity to unmodeled signals in generalized aliasing have conceptual similarity.
A Fourier representation reconstructed from uniformly spaced samples is useful, not only because it is dominated by low frequency modes, but also because the reconstructed signal is insensitive to high-frequency, unmodeled contributions.
In contrast, Fourier series from random samples exhibit strong sensitivity to unmodeled modes that distorts the coarse trends in the reconstructed signal.
In the language of renormalization, high-frequency modes are irrelevant for uniform samples, but random samples render the high-frequency modes relevant.

More recent work has informed similar conclusions about the general nature of predictive modeling.
Within the formalism of so-called \emph{sloppy models}, microscopic details of complicated, multi-parameter models can be safely ignored because observables of interest are insensitive to large variations in these parameters \cite{machta2013parameter,transtrum2015perspective, quinn2022information}.
Indeed, it has been found that many useful approximations may be derived by taking parameters to extreme values \cite{transtrum2014model}.
Even more fundamentally, evolutionary psychology has shown that psychological representations that maximize fitness are often not faithful to physical reality \cite{hoffman2015interface}.
That is, human psychological representations of reality may be dictated more by the sensitivity of evolutionary fitness to the representation than by fidelity to reality.
All of this suggests that when building a physical model, sensitivity to the unmodeled must be accounted for at least as much as fidelity to known physics.  

\subsection{Outlook}
\label{sec:outlook}

Successful model building involves numerous technical decisions related to the selection of model class, experimental design, learning algorithm, regularization, and other factors that can strongly impact the model's predictive performance.
Best practices are more often art, tuned to experience, rather than science guided by formal reasoning.
The generalized aliasing decomposition (Eq.~\eqref{eq:GAD}) facilitates reasoning about key modeling decisions in a way that is both formal and intuitive.
In the context of linear regression, the approach is fully rigorous while imbuing practitioners with intuition about model performance in both the classical and modern regimes.
Because the aliasing operator norm can be computed without knowing labels, practitioners can also make informed choices about data collection and experimental design for target applications.

Although our formal analysis has been restricted to linear regression, there are reasons to be optimistic that the core approach generalizes to the nonlinear regime.
First, the concepts of aliasing and invertibility (or projection to the kernel) extend formally to nonlinear operators and can be approximated through local linearization.
Furthermore, many cases of practical importance may be tractable in the present framework.
Results for weak, quadratic nonlinearities already exist for pseudospectral methods in partial differential equations \cite{orszag1971elimination}.
Neural tangent kernel techniques demonstrate that wide networks are linear in their models throughout training \cite{lee2020wide}.
In addition, information geometry techniques applied to large, sloppy models have shown that most nonlinearity is ``parameter-effects'' and removable, in principle, through an appropriate, nonlinear reparameterization \cite{transtrum2011geometry}.

An important open question is: Under what conditions is the asymptotic risk less than that of the classical ``sweet spot''?
The preceding analysis has sharpened that question to: When will data insufficiency be larger than the error at the classical sweet spot?
While this remains an open question in general, we have begun to explore it for two broad cases.  
Random feature models, such as in Figure~\ref{fig:belkin-MM-plot}, but presumably also neural networks and other machine learning models, often do not exhibit large data insufficiency and are generically most effective in the over-parameterized, modern regime.
In contrast, we have argued that physics-based models are most effective in the classical regime, where they leverage prior knowledge.

Framing the question in this way clarifies why classical statistics historically missed these interesting phenomena, in spite of the essential elements being known to diverse communities for decades \cite{LoogVieringMeyKrijtheTax2020}.
It also apparently partitions predictive modeling into two philosophically distinct camps: physical models using classical statistics and unstructured models in the modern, interpolating regime.
In our cluster expansion example, the former approach gave the model with the least risk.
Although perhaps expected, as physics-based modeling leverages prior information, this benefit comes after considerable effort from the materials science community.
However, it remains unclear if these are inherently irreconcilable philosophies or two points on a broad landscape just beginning to be explored.

Indeed, our work demonstrates how the theoretical and technical challenges posed by modern data science overlap with those in other fields, including signal processing, control theory, and statistical physics.
We hope that the perspectives advanced here will inspire theorists and practitioners alike to better understand and leverage the relationship between data science and the broader scientific milieu.

\acknowledgments{MKT was supported in part by the US NSF under awards DMR-1753357 and ECCS-2223985. 
GLWH was supported in part by the Chan-Zuckerberg Initiative's Imaging program.
JPW was partially supported by NSF grant DMS-2206762 and CCF-343286}


\bibliography{references} 




\end{document}